\documentclass[12pt,leqno,a4paper,final]{amsart}
\usepackage{amsmath, amsthm, verbatim, amsfonts, amssymb, color}
\usepackage{mathrsfs}
\usepackage{dsfont}
\usepackage[a4paper,margin=3cm]{geometry}
\usepackage{marginnote}
\usepackage[notcite,notref]{showkeys}

\theoremstyle{plain}
\newtheorem{theorem}{Theorem}[section]
\newtheorem{corollary}[theorem]{Corollary}
\newtheorem{lemma}[theorem]{Lemma}
\newtheorem{proposition}[theorem]{Proposition}
\theoremstyle{definition}
\newtheorem{remark}[theorem]{Remark}

\definecolor{wco}{rgb}{0.5,0.2,0.3}

\newcommand\Ee{\mathds E}
\newcommand\nat{\mathds N}
\newcommand\real{{\mathds{R}}}
\newcommand\I{\mathds 1}
\newcommand\Pp{\mathds P}

\newtheorem*{ack}{Acknowledgement}

\newcommand\Bscr{\mathscr{B}}

\newcommand\id{\operatorname{id}}

\newcommand\dup{\mathrm{d}}
\newcommand\eup{\mathrm{e}}
\newcommand\iup{\mathop{\mathrm{i}}}

\begin{document}
\title[Harnack Inequalities for SDEs]{\bfseries Harnack Inequalities for SDEs Driven by Time-Changed Fractional
Brownian Motions}

\author[C.-S.~Deng]{Chang-Song Deng}
\address[C.-S.~Deng]{School of Mathematics and Statistics\\ Wuhan University\\ Wuhan 430072, China}
\email{dengcs@whu.edu.cn}
\thanks{The first-named author gratefully acknowledges support through the Alexander-von-Humboldt foundation, and NNSFs of China (Nos.\ 11401442, 11371283). Part of this work was completed during his stay at TU Dresden as a Humboldt fellow. He is grateful for the hospitality and the excellent working conditions.}

\author[R.\,L.~Schilling]{Ren\'e L.\ Schilling}
\address[R.\,L.~Schilling]{TU Dresden\\ Fachrichtung Mathematik\\ Institut f\"{u}r Mathemati\-sche Stochastik\\ 01062 Dresden, Germany}
\email{rene.schilling@tu-dresden.de}



\maketitle

\begin{abstract}
    We establish Harnack inequalities for stochastic differential equations (SDEs) driven by a time-changed fractional Brownian motion with Hurst parameter $H\in(0,1/2)$. The Harnack inequality is dimension-free if the SDE has a drift which satisfies a one-sided Lipschitz condition; otherwise we still get Harnack-type estimates, but the constants will, in general, depend on the space dimension. Our proof is based on a coupling argument and a regularization argument for the time-change.
\end{abstract}

\section{Introduction}\label{intro}

Throughout this paper, $(\Omega,\mathscr{A},\Pp)$ is a
probability space. Consider the following $d$-dimensional SDE
\begin{equation}\label{intro-e01}
    X_t(x)
    =x+\int_0^tb_s\big(X_s(x)\big)\,\dup s+U_t,
    \quad t\geq 0,\;x\in\real^d,
\end{equation}
where $b:[0,\infty)\times\real^d\to\real^d$ is measurable, locally bounded in the time variable $t\geq 0$ and continuous in the space variable $x\in\real^d$; the driving noise $U=(U_t)_{t\geq 0}$ is a locally bounded measurable process on $\real^d$ starting at zero $U_0=0$. Let us assume, for the time being, that this SDE has a unique non-explosive solution.

In this paper, we want to establish for the solution to the SDE \eqref{intro-e01} a dimension-free Harnack inequality with power, first introduced by Wang \cite{Wan97} for diffusions on Riemannian manifolds, and a log-Harnack inequality, considered in \cite{RW10} for semi-linear SDEs. These
two Harnack-type inequalities have
many applications, for example
when studying the strong Feller property,
heat kernel estimates, contractivity properties,
entropy-cost inequalities, and many more; for an in-depth explanation we refer to the monograph by Wang \cite[Subsection 1.4.1]{Wbook} and the references
given there. Both, the power-Harnack and log-Harnack inequalities have been thoroughly investigated for
various finite- and infinite-dimensional SDEs
and SPDEs driven by Brownian noise; the main tool was a coupling method and the Girsanov transformation, see \cite{Wbook} and the references mentioned there. If the noise is a jump process, it is usually very difficult to construct a successful coupling, and the methods from diffusion processes cannot be directly
applied. One notable exception are driving noises which are subordinate to a diffusion process.

Let $\Sigma:[0,\infty)\to\real^d\otimes\real^d$ be a measurable and locally bounded deterministic function, and assume that $U$ is of the following form:
$$
    U_t=\int_0^t\Sigma_{s-}\,\dup W_{S(s)}+V_t,\quad t\geq 0,
$$
where $W=(W_t)_{t\geq 0}$ is a standard $d$-dimensional Brownian motion, $S=(S(t))_{t\geq 0}$ is a subordinator (i.e.\ a non-decreasing process on $[0,\infty)$ with stationary and independent increments a.k.a.\ increasing L\'evy process) and $V=(V_t)_{t\geq 0}$ is a locally
bounded ($\Bscr[0,\infty)\otimes\mathscr{A}/\Bscr(\real^d)$-)measurable process on $\real^d$ with $V_0=0$; we will, in addition, assume that the processes $W, S$ and $V$ are stochastically independent.

In this setting, Wang \& Wang \cite{WW14} were able to obtain Harnack and log-Harnack inequalities, using an approximation of the subordinator (as in \cite{Zha13}) and a coupling argument. The following assumptions turned out to be crucial: The coefficient $b$ has to satisfy a so-called one-sided Lipschitz condition, i.e.\ there exists a locally bounded measurable function $k:[0,\infty)\to\real$ such that
\begin{equation}\label{intro-e03}
    \langle b_t(x)-b_t(y),x-y\rangle
    \leq k(t)|x-y|^2,\quad x,y\in\real^d,\;t\geq 0;
\tag{\textup{H}}
\end{equation}
moreover, the inverse $\Sigma_t^{-1}$ exists for each $t\geq 0$, and there exists a non-decreasing function $\lambda:[0,\infty)\to[0,\infty)$ such that $\|\Sigma_t^{-1}\|\leq\lambda_t$ for all $t\geq 0$.

The first-named author used in \cite{Den14} the same approximation argument and a gradient estimate approach, in order to improve the Harnack inequalities derived in \cite{WW14}.
Recently, in \cite{WZ15} the approximation argument was also used to establish Harnack-type inequalities for SDEs with non-Lipschitz drift and anisotropic
subordinated Brownian noise, i.e.\ with $U$ having the form
$$
    U_t=\big(W^{(1)}_{S^{(1)}(t)},\dots,W^{(d)}_{S^{(d)}(t)}\big),
    \quad t\geq 0,
$$
where $\big(W^{(1)},\dots,W^{(d)}\big)$ is a standard Brownian motion in $\real^d$, and $\big(S^{(1)},\dots,S^{(d)}\big)$ is an independent $d$-dimensional L\'{e}vy process such that each coordinate process $S^{(i)}$ is a subordinator. Unfortunately, this gives only dimension-dependent Harnack
inequalities. Note that the techniques
of \cite{WW14, Den14, WZ15} do
not really need that the time-change is a subordinator; we may, as we do here, assume that the time-change is any non-decreasing process on $[0,\infty)$ starting from zero and which is independent of the original process.

It is a natural question to ask whether one can still get Harnack-type inequalities if the driving noise $U$ is a more general, maybe non-Markovian, process. As far as we know, Harnack inequalities were established in \cite{Fan13a, Fan13b, Fan14} for SDEs driven by fractional Brownian motions. Inspired by these papers as well as \cite{WW14, Den14}, we will combine general time-change and coupling arguments to obtain Harnack inequalities for SDEs driven by time-changed fractional Brownian motions.

Recall that a fractional Brownian motion $W^H=(W^H_t)_{t\geq 0}$ on $\real^d$ with Hurst parameter $H\in(0,1)$ is a self-similar, mean-zero Gaussian process with stationary increments. The covariance function is given by
\begin{equation}\label{intro-e05}
    \Ee\big(W^{H,(i)}_tW^{H,(j)}_s\big)
    =\frac12\left(t^{2H} +s^{2H}-|t-s|^{2H}\right)\delta_{ij},
    \quad t,s\geq 0,\; 1\leq i,j\leq d,
\end{equation}
($\delta_{ij}$ denotes Kronecker's delta). If $H=1/2$, then $W^H$ is the classical Brownian motion which will be denoted as $W$; if $H\neq1/2$, then $W^H$ does not have independent increments. One can deduce from \eqref{intro-e05} that $W^H$ is self-similar with index $H$, i.e.\ for any constant $c>0$, the processes $(W^H_{ct})_{t\geq 0}$ and $(c^HW^H_{t})_{t\geq 0}$ have the same finite dimensional distributions. Let $Z=(Z(t))_{t\geq 0}$ be a non-decreasing process on $[0,\infty)$ starting from $0$, independent of $W^H$, and introduce the (random) time-changed process $W^H_Z=(W^H_{Z(t)})_{t\geq 0}$. Typically, $Z$ can be a subordinator or the inverse of a subordinator; since inverse subordinators are constant on some random intervals, $W^H_Z$ is sometimes called a `delayed' fractional Brownian motion. We refer to \cite{LS04} for small deviation probabilities of time-changed fractional Brownian motions, while \cite{MNX08, GM14} consider large deviations of fractional Brownian motions delayed by inverse $\alpha$-stable subordinators.

Assume that $U=W^H_Z+V$ where $V$ is a locally bounded
measurable process on $\real^d$ starting
from zero $V_0=0$. In this paper, we restrict ourselves to the case $H\in(0,1/2)$. In order to ensure the existence and uniqueness of the solution to the SDE \eqref{intro-e01} and to construct a successful coupling, we assume that the coefficient $b$ satisfies the one-sided Lipschitz condition \eqref{intro-e03}. As a direct consequence of the log-Harnack inequality, we obtain a gradient estimate for the associated Markov operator.

As in \cite{WZ15}, we can also deal with the anisotropic case, i.e.
$$
    U_t
    =\big(W^{H_1,(1)}_{Z^{(1)}(t)},\dots,W^{H_d,(d)}_{Z^{(d)}(t)}\big) +V_t,\quad t\geq 0,
$$
where, for each $i=1,\dots, d$, $W^{H_i,(i)}=(W^{H_i,(i)}_t)_{t\geq 0}$ is a real-valued fractional Brownian motion with Hurst index $H_i\in(0,1/2)$, $Z^{(i)}=(Z^{(i)}(t))_{t\geq 0}$ is a one-dimensional non-decreasing process such that $Z^{(i)}(0)=0$, and $V=(V_t)_{t\geq 0}$ is a locally bounded measurable process with values in $\real^d$ and $V_0=0$; moreover, we assume that these processes are independent. As in \cite{WZ15}, we replace the Lipschitz condition for the drift coefficient $b$ by a Yamada--Watanabe-type condition; in general, however, this condition cannot be compared with the one-sided Lipschitz condition.

The remaining part of this paper is organized as follows. We collect some basics on fractional Brownian motions in Section \ref{bas}. In Section \ref{sde} we establish the Harnack inequalities for SDEs driven by a time-changed fractional Brownian motion and with drift coefficient satisfying the one-sided Lipschitz condition \eqref{intro-e03}. More explicit expressions in the Harnack and log-Harnack inequalities are obtained if the time-change $Z$ is (the inverse of) a subordinator; this is a consequence of our moment estimates from \cite{DS14}; if $Z$ is the inverse of a subordinator, only the log-Harnack inequality holds, since the exponential moment of $Z(t)^{-\theta}$ is usually infinite for $\theta>0$. The last section is devoted to the case of an anisotropic driving noise; as one would expect from \cite{WZ15}, the Harnack inequalities turn out to be dimension-dependent.

\section{Basics of fractional Brownian motion}\label{bas}

In this section, we recall briefly some basic facts on fractional Brownian motion (fBM) which will be used later on. For further details of fBM and proofs we refer the readers, for instance, to \cite{BHOZ08, DU98} or \cite{nourdin}.

Denote by $\Gamma(\cdot)$, resp., $B(\cdot,\cdot)$, the Euler Gamma and Beta functions, and write $\vphantom{F}_2F_1$ for Gauss' hypergeometric function.
Let $a,b\in\real$ with $a<b$. For $f\in L^1[a,b]$ and $\alpha>0$, the left fractional Riemann-Liouville integral of $f$ of order $\alpha$ on
$(a,b)$ is given by the expression
$$
    I_{a+}^\alpha f(x)
    :=\frac{1}{\Gamma(\alpha)} \int_a^x(x-y)^{\alpha-1}f(y)\,\dup y,
    \quad x\in (a,b).
$$
Let $W^H=(W^H_t)_{t\geq 0}$ be a fractional Brownian motion on $\real^d$ with Hurst parameter $H\in(0,1/2)\cup (1/2,1)$ and define for $0<s<t$ the kernel
\begin{align*}
    \mathcal{K}_H(t,s)
    :=&\frac{1}{\Gamma\left(H+ \tfrac 12\right)}(t-s)^{H- \frac 12} \vphantom{F}_2F_1\left(H-\tfrac12,\tfrac12-H,H+\tfrac12,1-\tfrac ts\right)
\end{align*}
Fix $T>0$. It is known that the operator
$\mathcal{K}_H$, associated with
the kernel $\mathcal{K}_H(\cdot,\cdot)$
$$
    \mathcal{K}_H  f^{(i)}(t)
    := \int_0^t \mathcal{K}_H(t,s)f^{(i)}(s)\,\dup s,
    \quad i=1,\dots,d,
$$
establishes a bijection from $L^2([0,T];\real^d)$
to the space $I_{0+}^{H+1/2}(L^2([0,T];\real^d))$,
see e.g.\ \cite[p.\ 187]{SKM93} or \cite{DU98}.
Moreover, fractional Brownian motion has the following integral representation with respect to a standard $d$-dimensional Brownian motion $W=(W_t)_{t\geq 0}$:
$$
    W^H_t=\int_0^t\mathcal{K}_H(t,s)\,\dup W_s.
$$

In particular, if $0<H<\frac 12$ and $h\in I_{0+}^{H+1/2}(L^2([0,T];\real^d))$ is absolutely continuous, the inverse operator is given by
\begin{equation}\label{bas-e03}
    (\mathcal{K}_H^{-1}h)(s)=s^{H-1/2}I_{0+}^{1/2-H}s^{1/2-H}h'(s),
\end{equation}
cf.\ \cite[Eq.\ (13), p.\ 108]{nua-ouk02}. If $H\in(0,1/2)$, then \eqref{bas-e03} implies
\begin{equation}\label{jg5fhnhg}
    \int_0^T|h(s)|^2\,\dup s<\infty\quad
    \Longrightarrow\quad
    \int_0^\bullet h(s)\,\dup s
    \in I_{0+}^{H+1/2}(L^2([0,T];\real^d)),
\end{equation}
see also \cite[p.\ 108]{nua-ouk02}.

\section{SDEs driven by delayed fractional Brownian motions}\label{sde}

Consider the following SDE on $\real^d$
\begin{equation}\label{sde-e03}
    X_t(x)
    =x+\int_0^tb_s\big(X_s(x)\big)\,\dup s+W^H_{Z(t)} + V_t,\quad t\geq 0,\; x\in\real^d,
\end{equation}
where $b:[0,\infty)\times\real^d\to\real^d$, $(t,x)\mapsto b_t(x)$ is measurable, locally bounded as a function of $t\geq 0$ and continuous in $x$. The processes $W^H=(W^H_t)_{t\geq 0}$, $Z=(Z_t)_{t\geq 0}$ and $V=(V_t)_{t\geq 0}$ are stochastically independent and satisfy
\begin{equation}\label{sde-e04}
\left\{\begin{aligned}
    &\text{$W^H$ is a fBM on $\real^d$ with Hurst parameter $H\in(0,1/2)$;}\\
    &\text{$Z$ is a time-change, i.e.\ a non-decreasing process on $[0,\infty)$ with $Z(0)=0$;}\\
    &\text{$V$ is a locally bounded measurable process on $\real^d$ with $V_0=0$.}
\end{aligned}\right.
\end{equation}
Moreover, we assume that the coefficient $b$ satisfies the one-sided Lipschitz condition \eqref{intro-e03}.

\begin{remark}\label{sde-01}
    The one-sided Lipschitz condition \eqref{intro-e03} ensures, in particular, the existence, uniqueness and non-explosion of the solution to the SDE \eqref{sde-e03}. Indeed, it is well known that the following ordinary differential equation
    $$
        Y_t(x)
        =x + \int_0^t\tilde{b}_s\left(Y_s(x)\right) \,\dup s,\quad t\geq 0,\; x\in\real^d,
    $$
    has a unique solution which does not explode in
    finite time since the coefficient $\tilde{b}$, defined by $\tilde{b}_s(\cdot):=b_s\left(\cdot+W^H_{Z(s)} + V_s\right)$, satisfies the one-sided Lipschitz condition \eqref{intro-e03} with $b$ replaced by $\tilde{b}$; setting $X_t(x):=Y_t(x)+W^H_{Z(t)} + V_t$, we conclude that the the SDE \eqref{sde-e03} has a unique non-explosive solution.
\end{remark}

Throughout this section, we write $|x|:=\big(|x^{(1)}|^2+\dots+|x^{(d)}|^2\big)^{1/2}$ for the Euclidean norm of $x=\big(x^{(1)},\dots,x^{(d)}\big)
\in\real^d$. Set
\begin{equation}\label{sde-e05}
    P_tf(x)
    := \Ee f(X_t(x)),
    \quad t\geq 0,\; f\in\Bscr_b(\real^d), \; x\in\real^d.
\end{equation}

\subsection{Statement of the main result}

In order to state our main result, we need the following notation:
\begin{equation}\label{sde-e11}
    K(t):=\int_0^t k(s)\,\dup s,\quad t\geq 0,
\end{equation}
where $k(s)$ is the constant appearing in \eqref{intro-e03},
\begin{equation}\label{sde-e13}
    \Theta_H
    :=\frac{1}{4(1-H)}\left(\frac{B\left(\frac32-H,\frac12-H\right)}{\Gamma\left(\frac12-H\right)}\right)^2,
\end{equation}
and we denote for any function $f:\real^d\to\real$ the local Lipschitz constant at the point $x$ by
$$
    |\nabla f|(x):=\limsup_{y\to x}
    \frac{|f(y)-f(x)|}{|y-x|}\in [0,\infty],\quad x\in\real^d.
$$

\begin{theorem}\label{sde-02}
    We assume that \eqref{sde-e04} and \eqref{intro-e03} hold for the SDE \eqref{sde-e03} and we denote its unique solution by $X_t(x)$.

\smallskip\noindent\textup{i)} \
        For $T>0$, $x,y\in\real^d$ and any bounded Borel function $f:\real^d \to [1,\infty)$
        $$
            P_T\log f(y)
            \leq
            \log P_Tf(x)+
            \Theta_H\Ee\left[
                \frac{Z(T)^{2-2H}}{\big( \int_0^T\eup^{-K(t)}\,\dup Z(t)\big)^2}
            \right]|x-y|^2.
        $$

\smallskip\noindent\textup{ii)} \
        For $T>0$, $x\in\real^d$ and any bounded Borel function $f:\real^d\to\real$
        $$
            |\nabla P_Tf|^2(x)
            \leq
            \left\{P_Tf^2(x)-\big(P_Tf(x)\big)^2\right\}
            2\Theta_H \Ee\left[
                \frac{Z(T)^{2-2H}}{\big(\int_0^T\eup^{-K(t)}\,\dup Z(t)\big)^2}
            \right].
        $$

\smallskip\noindent\textup{iii)} \
        For $T>0$, $x,y\in\real^d$, $p>1$ and any bounded Borel function $f:\real^d\to [0,\infty)$
        $$
            \big(P_Tf(y)\big)^p
            \leq
            P_Tf^p(x)\cdot
            \left(
                \Ee\exp\left[
                    \frac{\frac{p\Theta_H}{(p-1)^2} Z(T)^{2-2H} |x-y|^2}{\big(\int_0^T\eup^{-K(t)}\,\dup Z(t)\big)^2}
            \right]
            \right)^{p-1}.
        $$
\end{theorem}

\subsection{Proof of Theorem \ref{sde-02}}

For the proof of Theorem \ref{sde-02}, we need a few preparations. Let $\ell:[0,\infty)\to [0,\infty)$ be a non-decreasing
and c\`{a}dl\`{a}g function with $\ell(0)=0$, and $v:[0,\infty)\to\real$ a locally bounded measurable function with $v(0)=0$. By Remark~\ref{sde-01} the following SDE has a unique non-explosive solution
\begin{equation}\label{sde-e21}
    X_t^{\ell,v}(x)
    = x + \int_0^tb_s\big(X_s^{\ell,v}(x)\big)\,\dup s + W^H_{\ell(t)}+v_t,
    \quad t\geq 0,\; x\in\real^d.
\end{equation}
Set for any bounded Borel function $f:\real^d\to\real$
\begin{equation}\label{sde-e23}
    P_t^{\ell,v}f(x)
    :=\Ee f\big(X_t^{\ell,v}(x)\big),
    \quad t\geq 0,\; x\in\real^d.
\end{equation}

We want to transform the equation \eqref{sde-e21} into an SDE driven by a fractional Brownian motion which will allow us to establish Harnack inequalities using a combination of coupling and the Girsanov transformation, cf.\ \cite{Fan13a, Fan13b, Fan14}. First, however, we have to approximate the (deterministic) time-change $\ell$ by an absolutely continuous function. Consider the following regularization of $\ell$:
$$
    \ell_\epsilon(t):=\frac1\epsilon\int_t^{t+\epsilon}\ell(s)\,\dup s
    +\epsilon t=\int_0^1\ell(\epsilon s+t)\,\dup s+\epsilon t, \quad
    t\geq 0,\,\epsilon\in(0,1).
$$
By construction, for each $\epsilon\in (0,1)$ the function $\ell_\epsilon$ is absolutely continuous, strictly increasing and satisfies for any $t\geq0$
\begin{equation}\label{sde-e25}
    \ell_\epsilon(t)\downarrow\ell(t)
    \quad
    \text{as $\epsilon\downarrow 0$}.
\end{equation}
Let $X_t^{\ell_\epsilon,v}(x)$ be the unique non-explosive solution to the SDE
\begin{equation}\label{sde-e27}
    X_t^{\ell_\epsilon,v}(x)
    =x + \int_0^tb_s\big(X_s^{\ell_\epsilon,v}(x)\big) \,\dup s + W^H_{\ell_\epsilon(t)-\ell_\epsilon(0)} + v_t,
    \quad t\geq 0,\; x\in\real^d,
\end{equation}
and define $P^{\ell_\epsilon,v}$ by \eqref{sde-e23} with $\ell_\epsilon$ instead of $\ell$.

\begin{lemma}\label{sde-12}
    We assume that \eqref{sde-e04} and \eqref{intro-e03} hold for the SDE \eqref{sde-e03} and we denote its unique solution by $X_t(x)$.
    Fix $\epsilon\in (0,1)$ and let $\ell_\epsilon$ and $X_t^{\ell_\epsilon,v}(x)$ be as above.

\smallskip\noindent\textup{i)} \
        For $T>0$, $x,y\in\real^d$ and any bounded Borel function $f:\real^d\to [1,\infty)$
        $$
            P_t^{\ell_\epsilon,v}\log f(y)
            \leq
            \log P_t^{\ell_\epsilon,v}f(x)+
            \frac{\Theta_H [\ell_\epsilon(T)-\ell_\epsilon(0)]^{2-2H}}{\big(\int_0^T\eup^{-K(t)}\,\dup\ell_\epsilon(t)
            \big)^2} |x-y|^2.
        $$

\smallskip\noindent\textup{ii)} \
        For $T>0$, $x,y\in\real^d$, $p>1$ and any bounded Borel function $f:\real^d\to [0,\infty)$
        $$
            \big(P_t^{\ell_\epsilon,v}f(y)\big)^p
            \leq
            P_t^{\ell_\epsilon,v}f^p(x)\cdot
            \exp\left[
                \frac{\frac{p\Theta_H}{p-1} [\ell_\epsilon(T)-\ell_\epsilon(0)]^{2-2H}|x-y|^2}{
                \big(\int_0^T\eup^{-K(t)}\,\dup\ell_\epsilon(t)\big)^2}\right].
        $$
\end{lemma}

\begin{proof}
    Fix $T>0$, $x,y\in\real^d$ and denote
    by $(Y_t)_{t\geq 0}$ a solution of the equation
    \begin{equation}\label{sde-e29}
        Y_t
        =
        y + \int_0^tb_s(Y_s)\,\dup s + W^H_{\ell_\epsilon(t)-\ell_\epsilon(0)} + v_t +
        \xi\int_0^t\I_{[0,\tau)}(s) \frac{X_s^{\ell_\epsilon,v}(x)-Y_s} {|X_s^{\ell_\epsilon,v}(x)-Y_s|}\,\dup\ell_\epsilon(s),
    \end{equation}
    where
    $$
        \xi
        :=\frac{|x-y|}{\int_0^T\eup^{-K(r)}\,\dup\ell_\epsilon(r)}
    $$
    and
    $$
        \tau
        :=\inf\big\{t\geq 0\,:\,X_t^{\ell_\epsilon,v}(x)=Y_t\big\}
    $$
    is the coupling time. Since
    $$
        \real^d\times\real^d\ni(z,z')\mapsto\I_{\{z\neq z'\}}\frac{z-z'}{|z-z'|}\in\real^d
    $$
    is locally Lipschitz continuous off the diagonal,
    the system of coupled equations \eqref{sde-e27} and \eqref{sde-e29} has a unique
    solution for $t\in[0,\tau)$.
    If $\tau<\infty$, we set $Y_t=X_t^{\ell_\epsilon,v}(x)$ for all $t\geq\tau$. In this way, we can construct a unique solution $(Y_t)_{t\geq 0}$ to \eqref{sde-e29}.

    Let us show that the coupling time satisfies $\tau\leq T$. Let $t<\tau$, write $\zeta_t$ for the difference of the solutions to the SDEs \eqref{sde-e27} and \eqref{sde-e29}, and observe that $\zeta_t$ admits a classic differential satisfying
    $\dup|\zeta_t|=\I_{\{\zeta_t\neq 0\}}|\zeta_t|^{-1}
    \langle\zeta_t,\dup\zeta_t\rangle$; therefore, \eqref{intro-e03} yields
    \begin{align*}
        &|X_t^{\ell_\epsilon,v}(x)-Y_t|\eup^{-K(t)}\\
        &=|x-y| + \int_0^t\frac{1}{|X_s^{\ell_\epsilon,v}(x)-Y_s|}
        \left\langle X_s^{\ell_\epsilon,v}(x)-Y_s,\, b_s(X_s^{\ell_\epsilon,v}(x)) - b_s(Y_s)\right\rangle \eup^{-K(s)}\,\dup s\\
        &\qquad\mbox{} -\xi\int_0^t\eup^{-K(s)} \,\dup\ell_\epsilon(s) -\int_0^t |X_s^{\ell_\epsilon,v}(x)-Y_s| \, k(s)\eup^{-K(s)}\,\dup s\\
        &\leq|x-y|-\xi\int_0^t \eup^{-K(s)}\,\dup\ell_\epsilon(s).
    \end{align*}
    Now assume that $\tau(\omega)>T$ for some
    $\omega\in\Omega$. Taking $t=T$ in the
    above inequality, we get
    $$
        0<|X_T^{\ell_\epsilon,v}(x,\omega)-Y_T(\omega)|  \eup^{-K(T)}
        \leq|x-y|-|x-y|=0,
    $$
    which is absurd. Therefore, we have $\tau\leq T$ and     $X_T^{\ell_\epsilon,v}(x)=Y_T$.

    Denote by $\gamma_\epsilon : [\ell_\epsilon(0),\infty) \to [0,\infty)$
    the inverse function of $\ell_\epsilon$.
    By definition, $\ell_\epsilon\left(\gamma_\epsilon(t)\right)=t$
    for $t\geq\ell_\epsilon(0)$,
    $\gamma_\epsilon\left(\ell_\epsilon(t)\right)=t$
    for $t\geq 0$, and $t\mapsto\gamma_\epsilon(t)$ is
    absolutely continuous and strictly
    increasing. Let
    $$
        g_r
        :=
        \xi \I_{[0,\ell_\epsilon(\tau)-\ell_\epsilon(0))}(r)
        \frac{X_{\gamma_\epsilon(r+\ell_\epsilon(0))}
        ^{\ell_\epsilon,v}(x)-
        Y_{\gamma_\epsilon(r+\ell_\epsilon(0))}}
        {|X_{\gamma_\epsilon(r+\ell_\epsilon(0))}
        ^{\ell_\epsilon,v}(x)-
        Y_{\gamma_\epsilon(r+\ell_\epsilon(0))}|},
        \quad r\geq 0.
    $$
    A simple calculation shows
    $\int_0^{\ell_\epsilon(T)-\ell_\epsilon(0)}|g_r|^2\,
    \dup r<\infty$, and this, together with
    $H<1/2$ and \eqref{jg5fhnhg}, implies
    that $\int_0^\bullet g_r\,\dup r\in I_{0+}^{H+1/2}(L^2([0,\ell_\epsilon(T)-\ell_\epsilon(0)];
    \real^d))$. Therefore, the following stochastic
    integral defines a martingale
    $$
        M_t
        := -\int_0^{t}\langle \eta_s,\dup W_s\rangle,
        \quad t\geq 0,
    $$
    where $\eta_s:=\mathcal{K}_H^{-1}\left(\int_0^\bullet g_r\,\dup r\right)(s)$, $s\geq 0$, and
    $W=(W_t)_{t\geq 0}$ is a $d$-dimensional
    standard Brownian motion.
    Because of \eqref{bas-e03} we see
    $$
        \mathcal{K}_H^{-1}\left(\int_0^\bullet g_r\,\dup r\right)(s)
        = s^{H-\frac12}I_{0+}^{\frac12-H}s^{\frac12-H}g_s,
    $$
    and this yields for any $s\in[0,\ell_\epsilon(T)-\ell_\epsilon(0)]$
    \begin{align*}
        |\eta_s|
        &= \left|\frac{1}{\Gamma\left(\frac12-H\right)} s^{H-\frac12} \int_0^s r^{\frac12-H}(s-r)^{-H-\frac12}g_r\,\dup r\right|\\
        &\leq \frac{\xi}{\Gamma\left(\frac12-H\right)} s^{H-\frac12} \int_0^s r^{\frac12-H}(s-r)^{-H-\frac12}\,\dup r\\
        &= \frac{B\left(\frac32-H,\frac12-H\right)}{\Gamma\left(\frac12-H\right) \int_0^T\eup^{-K(t)}\,\dup\ell_\epsilon(t)}\,|x-y|\, s^{\frac12-H}\\
        &=: C_{T,H}|x-y|\,s^{\frac12-H}.
    \end{align*}
    Thus, the compensator of the martingale $M$ satisfies
    \begin{equation}\label{sde-e31}
    \begin{aligned}
        \langle M\rangle_{\ell_\epsilon(T)-\ell_\epsilon(0)}
        &=\int_0^{\ell_\epsilon(T)-\ell_\epsilon(0)}|\eta_s|^2\,\dup s\\
        &\leq C_{T,H}^2|x-y|^2\int_0^{\ell_\epsilon(T)-\ell_\epsilon(0)} s^{1-2H}\,\dup s\\
        &=\frac{C_{T,H}^2|x-y|^2}{2(1-H)}\left[\ell_\epsilon(T)-\ell_\epsilon(0)\right]^{2-2H}.
    \end{aligned}
    \end{equation}
    Set
    $$
        R:=\exp\left[M_{\ell_\epsilon(T)-\ell_\epsilon(0)}
        -\frac12\langle M\rangle _{\ell_\epsilon(T)-\ell_\epsilon(0)}\right].
    $$
    Since $\Ee\,\eup^{\frac12\langle M\rangle _{\ell_\epsilon(T)-\ell_\epsilon(0)}}<\infty$,
    one can use Novikov's criterion to get $\Ee R=1$,
    and by Girsanov's theorem, the process
    $$
        \widetilde{W}_t:=W_t+\int_0^t\eta_s\,\dup s,
        \quad 0\leq t\leq \ell_\epsilon(T)
        -\ell_\epsilon(0),
    $$
    is a $d$-dimensional Brownian motion under the new probability measure $R\Pp$. This allows us
    to rewrite \eqref{sde-e27} and \eqref{sde-e29} as
    $$
        X_t^{\ell_\epsilon,v}(x)
        =x + \int_0^tb_s\big(X_s^{\ell_\epsilon,v}(x)\big)\,\dup s + \int_0^{\ell_\epsilon(t)-\ell_\epsilon(0)} \mathcal{K}_H\left(\ell_\epsilon(t)-\ell_\epsilon(0), s\right)\,\dup W_s + v_t
    $$
    and
    $$
        Y_t
        = y + \int_0^tb_s(Y_s)\,\dup s + \int_0^{\ell_\epsilon(t)-\ell_\epsilon(0)} \mathcal{K}_H\left(\ell_\epsilon(t)-\ell_\epsilon(0), s\right)\,\dup \widetilde{W}_s + v_t,
    $$
    respectively. Thus, the distribution of $(X_T^{\ell_\epsilon,v}(y))_{0\leq t\leq T}$ under $\Pp$ coincides with the law of $(Y_t)_{0\leq t\leq T}$ under $R\Pp$; in particular, we get for all bounded Borel functions $f:\real^d\to\real$
    \begin{equation}\label{sde-e33}
        \Ee f\big(X_T^{\ell_\epsilon,v}(y)\big)
        =\Ee_{R\Pp}f(Y_T)
        =\Ee\left[Rf(Y_T)\right]
        =\Ee\big[Rf\big(X_T^{\ell_\epsilon,v}(x)\big)\big].
    \end{equation}
    By the Jensen inequality, we get for any random variable $F\geq1$,
    $$
        \Ee\left[R\log\frac{F}{R}\right]
        =\Ee_{R\Pp}\left[\log\frac{F}{R}\right]
        \leq\log\Ee_{R\Pp}\left[\frac{F}{R}\right]
        =\log\Ee F,
    $$
    hence
    $$
        \Ee\left[R\log F\right]
        \leq\log\Ee F + \Ee\left[R\log R\right].
    $$
    Combining this with \eqref{sde-e33} and the observation that
    \begin{align*}
        \log R
        &= -\int_0^{\ell_\epsilon(T)-\ell_\epsilon(0)}\langle \eta_s,\dup W_s\rangle-\frac12\int_0^{\ell_\epsilon(T)-\ell_\epsilon(0)}|\eta_s|^2\,\dup s\\
        &= -\int_0^{\ell_\epsilon(T)-\ell_\epsilon(0)}\langle \eta_s,\dup \widetilde{W}_s\rangle+\frac12\langle M\rangle_{\ell_\epsilon(T)-\ell_\epsilon(0)}\\
        &\leq-\int_0^{\ell_\epsilon(T)-\ell_\epsilon(0)}\langle \eta_s,\dup \widetilde{W}_s\rangle+\frac{C_{T,H}^2|x-y|^2}{4(1-H)}
        \left[\ell_\epsilon(T)-\ell_\epsilon(0)\right]^{2-2H},
    \end{align*}
    we get for all bounded Borel functions $f:\real^d\to [1,\infty)$ that
    \begin{align*}
        P_T^{\ell_\epsilon,v}\log f(y)
        &= \Ee\log f\big(X_T^{\ell_\epsilon,v}(y)\big)\\
        &= \Ee\big[R\log f\big(X_T^{\ell_\epsilon,v}(x)\big)\big]\\
        &\leq \log\Ee f\big(X_T^{\ell_\epsilon,v}(x)\big)+\Ee [R\log R]\\
        &= \log P_T^{\ell_\epsilon,v}f(x)+\Ee_{R\Pp}\log R\\
        &\leq \log P_T^{\ell_\epsilon,v}f(x)+ \frac{C_{T,H}^2|x-y|^2}{4(1-H)}\left[\ell_\epsilon(T)-\ell_\epsilon(0)\right]^{2-2H}.
    \end{align*}
    This completes the proof of the log-Harnack inequality.

    \medskip
    Let us now prove part ii) of the Lemma. For any bounded Borel function $f:\real^d\to[0,\infty)$ we find with \eqref{sde-e33} and the H\"{o}lder inequality
    \begin{equation}\label{sde-e35}
    \begin{aligned}
        \big(P_T^{\ell_\epsilon,v}f(y)\big)^p
        &=\big(\Ee f\big(X_T^{\ell_\epsilon,v}(y)\big)\big)^p\\
        &=\big(\Ee\big[R f\big(X_T^{\ell_\epsilon,v}(x)\big)\big]\big)^p\\
        &\leq P_T^{\ell_\epsilon,v}f^p(x)\cdot
        \left(\Ee[R^{p/(p-1)}]\right)^{p-1}.
    \end{aligned}
    \end{equation}
    Using \eqref{sde-e31} we get
    \begin{align*}
        R^{p/(p-1)}
        &= \exp\left[\frac{p}{p-1}M_{\ell_\epsilon(T)-\ell_\epsilon(0)}-\frac{p}{2(p-1)}\langle M\rangle_{\ell_\epsilon(T)-\ell_\epsilon(0)}\right]\\
        &= \exp\left[\frac{p}{2(p-1)^2}\langle M\rangle_{\ell_\epsilon(T)-\ell_\epsilon(0)}\right]\\
        &\qquad\qquad\mbox{}\times
           \exp\left[\frac{p}{p-1}M_{\ell_\epsilon(T)-\ell_\epsilon(0)}-\frac{p^2}{2(p-1)^2}\langle M\rangle_{\ell_\epsilon(T)-\ell_\epsilon(0)}\right]\\
        &\leq \exp\left[\frac{pC_{T,H}^2|x-y|^2}{4(p-1)^2(1-H)}\left[\ell_\epsilon(T)-\ell_\epsilon(0)\right]^{2-2H}\right]\\
        &\qquad\qquad\mbox{}\times
              \exp\left[\frac{p}{p-1}M_{\ell_\epsilon(T)-\ell_\epsilon(0)}-\frac{p^2}{2(p-1)^2}\langle M\rangle_{\ell_\epsilon(T)-\ell_\epsilon(0)}\right].
    \end{align*}
    Noting the fact that
    $\exp\left[\frac{p}{p-1} M_{\ell_\epsilon(t)-\ell_\epsilon(0)}- \frac{p^2}{2(p-1)^2}\langle M\rangle_{\ell_\epsilon(t)-\ell_\epsilon(0)}\right]$, $0\leq t\leq T$, is a martingale with mean $1$ -- this is due to Novikov's criterion -- we get
    $$
        \Ee\left[R^{p/(p-1)}\right]
        \leq
        \exp\left[\frac{pC_{T,H}^2|x-y|^2}{4(p-1)^2(1-H)}\left[\ell_\epsilon(T)-\ell_\epsilon(0)\right]^{2-2H}\right].
    $$
    Inserting this expression into \eqref{sde-e35}, completes the proof of the power-Harnack inequality.
\end{proof}

\begin{proof}[Proof of Theorem \ref{sde-02}]
    By \cite[Proposition 2.3]{ATW14}, ii) is a direct
    consequence of i).

    Fix $T>0$. By a standard approximation argument, it is enough to prove the formulae in i) and iii) for $f\in C_b(\real^d)$.

    \emph{Step 1:} Assume that $b_t:\real^d\to\real^d$ is, uniformly for $t$ in compact intervals, a global Lipschitz function, i.e.\ for any
    $t>0$ there is some $C_t>0$ such that
    \begin{equation}\label{sde-e41}
        |b_s(x)-b_s(y)| \leq C_t|x-y|,
        \quad 0\leq s\leq t,\; x,y\in\real^d.
    \end{equation}
    This implies that for all $x\in\real^d$ and $\epsilon\in(0,1)$
    \begin{align*}
        \big|X_T^{\ell_\epsilon,v}(x)-X_T^{\ell,v}(x)\big|
        &\leq
            \int_0^T    \big|b_s\big(X_s^{\ell_\epsilon,v}(x)\big) - b_s\big(X_s^{\ell,v}(x)\big)\big|\,\dup s
            +\big|W^H_{\ell_\epsilon(T)-\ell_\epsilon(0)} -W^H_{\ell(T)}\big|\\
        &\leq
            C_T\int_0^T \big|X_s^{\ell_\epsilon,v}(x)-X_s^{\ell,v}(x)\big|\,\dup s
            +\big|W^H_{\ell_\epsilon(T)-\ell_\epsilon(0)} - W^H_{\ell(T)}\big|.
    \end{align*}
    Since $X^{\ell_\epsilon,v}_t(x)$ and $X^{\ell,v}_t(x)$ are non-explosive, the integral in the above expression is finite. Therefore, we can apply Gronwall's inequality with $g(\epsilon,t):=\big|W^H_{\ell_\epsilon(t)-\ell_\epsilon(0)} -W^H_{\ell(t)}\big|$ and find
    $$
        \big|X_T^{\ell_\epsilon,v}(x)-X_T^{\ell,v}(x)\big|
        \leq
        g(\epsilon,T)+ C_T\int_0^T g(\epsilon,s) \eup^{(T-s)C_T}\, \dup s.
    $$
    From \eqref{sde-e25}, we conclude that $\lim_{\epsilon\downarrow 0}g(\epsilon,s)=0$ for all $s\geq0$. Using the dominated convergence
    theorem, we obtain
    $$
        \lim_{\epsilon\downarrow 0}X_T^{\ell_\epsilon,v}(x)
        =X_T^{\ell,v}(x),\quad x\in\real^d;
    $$
    hence,
    $$
        \lim_{\epsilon\downarrow 0}P_T^{\ell_\epsilon,v}f
        =P_T^{\ell,v}f,\quad f\in C_b(\real^d).
    $$
    Since $\ell$ is of bounded variation,
    the limit $\ell_\epsilon\downarrow\ell$ also
    holds for the integrals
    $$
        \lim_{\epsilon\downarrow 0} \int_0^T\eup^{-K(t)}\,\dup\ell_\epsilon(t)
        = \int_0^T\eup^{-K(t)}\,\dup\ell(t).
    $$
    We can now use Lemma \ref{sde-12} i) and let $\epsilon\downarrow 0$ to get
    $$
            P_T^{\ell,v}\log f(y)
            \leq \log P_T^{\ell,v}f(x) + \frac{\Theta_H \ell(T)^{2-2H}}{\big(\int_0^T\eup^{-K(t)}\,\dup\ell(t)\big)^2}\,|x-y|^2
    $$
    for $x,y\in\real^d$ and all $f\in C_b(\real^d)$ with $f\geq1$. Similarly, Lemma \ref{sde-12} ii) yields
    $$
            \big(P_T^{\ell,v}f(y)\big)^p
            \leq P_T^{\ell,v}f^p(x)\cdot
            \exp\left[\frac{\frac{p}{p-1} \Theta_H \ell(T)^{2-2H}}{\big(\int_0^T\eup^{-K(t)}\,\dup\ell(t)\big)^2}\,|x-y|^2\right]
    $$
    for $x,y\in\real^d$ and all non-negative $f\in C_b(\real^d)$.

    Since the processes $W$, $V$ and $Z$ are independent,
    $
        P_Tf
        =\Ee\left[P_T^{\ell,v}f(\cdot)\left|_{\begin{subarray}{c}\ell=Z\\ v=V\end{subarray}}\right.
        \right]
    $
    holds for all bounded Borel functions $f:\real^d\to\real$. Thus, the Jensen inequality yields for all $x,y\in\real^d$ and $f\in C_b(\real^d)$ with $f\geq1$
    \begin{align*}
        P_T\log f(y)
        &= \Ee\left[P_T^{\ell,v}\log f(y)\left|_{\begin{subarray}{c}\ell=Z\\ v=V\end{subarray}}\right.\right]\\
        &\leq \Ee\left[\log P_T^{\ell,v} f(x)\left|_{\begin{subarray}{c}\ell=Z\\ v=V\end{subarray}}\right.\right]
        + \Theta_H \Ee\left[ \frac{Z(T)^{2-2H}}{\big(\int_0^T\eup^{-K(t)}\,\dup Z(t)\big)^2}\right]|x-y|^2\\
        &\leq \log P_Tf(x) + \Theta_H \Ee\left[ \frac{Z(T)^{2-2H}}{\big(\int_0^T\eup^{-K(t)}\,\dup Z(t)\big)^2}\right]|x-y|^2.
    \end{align*}
    For the power-Harnack inequality we use H\"{o}lder's inequality to find for all $x,y\in\real^d$ and non-negative $f\in C_b(\real^d)$
    \begin{align*}
        P_Tf(y)
        &= \Ee\left[P_T^{\ell,v}f(y)\left|_{\begin{subarray}{c}\ell=Z\\ v=V\end{subarray}}\right.\right]\\
        &\leq
        \Ee\left[\left.
            \big(P_T^{\ell,v}f^p(x)\big)^{\frac 1p}
            \exp\left[ \frac{\frac{\ell(T)^{2-2H}}{p-1}\Theta_H |x-y|^2}{\big(\int_0^T\eup^{-K(t)}\,\dup\ell(t)\big)^2} \right]
        \right|_{\begin{subarray}{c}\ell=Z\\ v=V\end{subarray}}\right]\\
        &\leq
        (P_Tf^p(x)\big)^{\frac 1p}
        \left(\Ee\exp\left[\frac{\frac{pZ(T)^{2-2H}}{(p-1)^2}\Theta_H |x-y|^2}{\big(\int_0^T\eup^{-K(t)}\,\dup Z(t)\big)^2}\right]\right)^{1-\frac 1p}.
    \end{align*}

\medskip\emph{Step 2:} For the general case, we use the approximation argument proposed in \cite[part (c) of proof of Theorem 2.1]{WW14}.
    Let
    $$
        \tilde{b}_t(x):=b_t(x)-k(t)x,
        \quad t\geq 0,\; x\in\real^d.
    $$
    ($k(t)$ is the constant appearing in
    \eqref{intro-e03} on p.~\pageref{intro-e03}.) Using \eqref{intro-e03}, it is not difficult to see that the mapping $\id-n^{-1}\tilde{b}_t:\real^d\to\real^d$ is injective for any $n\in\nat$ and $t\geq 0$. The maps
    $$
        b_t^{(n)}(x)
        :=n\left[\left(\id-n^{-1}\tilde{b}_t\right)^{-1}(x)-x\right]+k(t)x,
        \quad n\in\nat,\; t\geq 0,\; x\in\real^d,
    $$
    are, uniformly for $t$ in compact intervals, globally Lipschitz continuous, see \cite{DRW09}. Denote by $(X^{(n)}_t(x))_{t\geq 0}$ the solution of \eqref{sde-e03} with $b$ replaced by $b^{(n)}$, and define $P_t^{(n)}$ by \eqref{sde-e05} with $X_t(x)$ replaced by $X^{(n)}_t(x)$. Because of the first part of the proof, the statements of Theorem \ref{sde-02} hold with $P_T$ replaced by $P_T^{(n)}$.

    On the other hand, we see as in \cite[part (c) of proof of Theorem 2.1]{WW14}, that
    $$
        \lim_{n\to\infty}X_T^{(n)}(x)=X_T(x)\,
        \text{\ a.s.,}\quad \text{hence},\quad
        \lim_{n\to\infty}P_T^{(n)}f =P_Tf
        \quad\text{for all $f\in C_b(\real^d)$}.
    $$
    Therefore, the claim follows if we let $n\to\infty$.
\end{proof}

\subsection{Applications}

Let $Z=(Z(t))_{t\geq 0}$ be a
subordinator (without killing), i.e.\ a nondecreasing
L\'{e}vy process on $[0,\infty)$ with $S_0=0$; its
Laplace transform is of the form
$$
    \Ee\,\eup^{-rZ(t)}
    = \eup^{-t\phi(r)},\quad r>0,\; t\geq 0,
$$
and the characteristic (Laplace) exponent $\phi:(0,\infty)\rightarrow(0,\infty)$
is a Bernstein function with $\phi(0+)=0$. Recall that a Bernstein function is a smooth function $\phi\in C^\infty((0,\infty))$ such that $(-1)^{n-1}\phi^{(n)}\geq0$
for all $n=1,2,\dots$; it is well known,
see e.g.\ \cite[Theorem 3.2]{SSV}, that every
Bernstein function enjoys a unique L\'{e}vy--Khintchine
representation
$$
    \phi(r)
    =\phi(0+) + \vartheta r + \int_{(0,\infty)}\left(1-\eup^{-rx}\right)\,\nu(\dup x),
    \quad r>0,
$$
where $\vartheta\geq 0$ is the drift parameter and $\nu$ is a L\'{e}vy measure, that is, a measure on $(0,\infty)$ satisfying $\int_{(0,\infty)}(1\wedge x)\,\nu(\dup x)<\infty$.

For the constant $k(t)$ from \eqref{intro-e03} and its primitive $K(t)$, cf.\ \eqref{sde-e11}, we set
$$
    K^*(T):=\exp\left[2\sup_{t\in[0,T]}K(t)\right],\quad T>0.
$$
Obviously, if $k(t)\leq 0$ for all $t\geq 0$, then $K^*(T)\leq1$ for all $T>0$.

\begin{corollary}\label{sde-52}
    Let $Z$ be a subordinator whose characteristic exponent is the Bernstein function $\phi$ and assume that \eqref{intro-e03} and \eqref{sde-e04} hold. We have for all $T>0$, $x,y\in\real^d$ and all bounded Borel functions $f:\real^d\to[0,\infty)$ the following assertions:

\smallskip\noindent\textup{i)} \
        If $\liminf_{r\to\infty}\phi(r)r^{-\rho}>0$ for some
        $\rho>0$, then
        \begin{gather*}
            P_T\log f(y)
            \leq
            \log P_Tf(x) + \frac{C_{H,\rho}K^*(T)|x-y|^2}{(T\wedge1)^{2H/\rho}},
            \quad f\geq 1,\\
            |\nabla P_Tf|^2(x)
            \leq
            \left\{P_Tf^2(x)-\big(P_Tf(x)\big)^2\right\}
            \frac{C_{H,\rho}K^*(T)}{(T\wedge1)^{2H/\rho}}.
        \end{gather*}
        If, in addition,
        $\liminf_{r\downarrow 0}\phi(r)r^{-\rho}>0$,
        then we can replace $T\wedge 1$ by $T$ and get
        \begin{gather*}
            P_T\log f(y)
            \leq
            \log P_Tf(x) + \frac{C_{H,\rho}K^*(T)|x-y|^2}{T^{2H/\rho}},
            \quad f\geq 1,\\
            |\nabla P_Tf|^2(x)
            \leq
            \left\{P_Tf^2(x)-\big(P_Tf(x)\big)^2\right\}
            \frac{C_{H,\rho}K^*(T)}{T^{2H/\rho}}.
        \end{gather*}

\smallskip\noindent\textup{ii)} \
        If $\liminf_{r\to\infty}\phi(r)r^{-\rho}>0$ for some $\rho>2H/(1+2H)$ and $p>1$, then
        \begin{align*}
            &\big(P_Tf(y)\big)^p\leq
            P_Tf^p(x)\\
            &\quad\times
            \exp\left[
            \frac{C_{H,\rho}\,p}{p-1}
            K^*(T)
            \left(1+\frac{1}{T^{2H/\rho}}
            \right)|x-y|^2
            +C_{H,\rho}
            \frac{\left(pK^*(T)|x-y|^2\right)^
            {\frac{\rho}{\rho-2H(1-\rho)}}
            }
            {(p-1)^{\frac{\rho+2H(1-\rho)}
            {\rho-2H(1-\rho)}}
            T^{\frac{2H}{\rho-2H(1-\rho)}}}
            \right].
        \end{align*}
        If, in addition,
        $\liminf_{r\downarrow 0}\phi(r)r^{-\rho}>0$, then
        \begin{align*}
            &\big(P_Tf(y)\big)^p\leq
            P_Tf^p(x)\\
            &\qquad\qquad\qquad\quad\times
            \exp\left[
            \frac{C_{H,\rho}\,p}{p-1}
            \frac{K^*(T)
            |x-y|^2}{T^{2H/\rho}}
            +C_{H,\rho}
            \frac{\left(pK^*(T)|x-y|^2\right)^
            {\frac{\rho}{\rho-2H(1-\rho)}}
            }
            {(p-1)^{\frac{\rho+2H(1-\rho)}
            {\rho-2H(1-\rho)}}
            T^{\frac{2H}{\rho-2H(1-\rho)}}}
            \right].
        \end{align*}
\end{corollary}

\begin{proof}
    Since we have
    \begin{equation}\label{sde-e57}
        Z(T)^{2-2H}\left(\int_0^T\eup^{-K(t)}\,\dup Z(t)\right)^{-2}
        \leq
        K^*(T) Z(T)^{-2H},\quad T>0,
    \end{equation}
    the assertion follows from Theorem \ref{sde-02} and \cite[Theorem 3.8\,(a) and (b)]{DS14}.
\end{proof}

We will now assume that the subordinator $S$ is strictly increasing, i.e.\ we have $\nu(0,\infty)=\infty$ or $\vartheta>0$. Define the (generalized, right-continuous) inverse of $S$
$$
    S^{-1}(t)
    :=\inf\{s\geq 0\,:\,S(s)> t\}
    =\sup\{s\geq 0\,:\,S(s)\leq t\},\quad t\geq 0.
$$
We will call $S^{-1}=(S^{-1}(t))_{t\geq 0}$ an inverse subordinator associated with the Bernstein function $\phi$. Since we assume that the subordinator
$S$ is strictly increasing, we know that almost all paths of $S^{-1}$ are continuous and non-decreasing. We will frequently use the following identity:
\begin{equation}\label{sde-e59}
    \Pp\left(S(r)\geq t\right)
    =\Pp\left(S^{-1}(t)\leq r\right),
    \quad r,t>0.
\end{equation}

\begin{corollary}\label{sde-54}
    Let $Z$ be an inverse subordinator associated with the Bernstein function $\phi$ and assume that
    \eqref{intro-e03} and \eqref{sde-e04} hold.

    If $\limsup_{r\downarrow 0}\phi(r)r^{-\sigma}<\infty$ and $\limsup_{r\to\infty}\phi(r) r^{-\sigma}<\infty$ for some $\sigma>0$, then the following assertions hold.

    \smallskip\noindent\textup{i)} \
        For any $T>0$, $x,y\in\real^d$ and all bounded Borel functions $f:\real^d\to [1,\infty)$
        $$
            P_T\log f(y)
            \leq\log P_Tf(x) + \frac{C_{H,\sigma}K^*(T)}{T^{2H\sigma}} |x-y|^2.
        $$

    \smallskip\noindent\textup{ii)} \
        For any $T>0$, $x\in\real^d$ and all bounded Borel functions $f:\real^d\to [0,\infty)$
        $$
            |\nabla P_Tf|^2(x)
            \leq
            \left\{P_Tf^2(x)-\big(P_Tf(x)\big)^2\right\}\frac{C_{H,\sigma}K^*(T)}{T^{2H\sigma}}.
        $$
\end{corollary}

Corollary \ref{sde-54} follows, if we combine \eqref{sde-e57} with Theorem \ref{sde-02} i), ii) and the next lemma.
\begin{lemma}\label{sde-58}
    Let $S^{-1}$ be an inverse subordinator with Bernstein function $\phi$ satisfying the conditions of Corollary \ref{sde-54}. For any $\theta\in(0,1)$,
    $$
        \Ee\left[\left(S^{-1}(t)\right)^{-\theta}\right]
        \leq C_{\sigma,\theta}\,t^{-\sigma\theta},\quad t>0.
    $$
\end{lemma}
\begin{proof}
        By our assumption, there exists a constant $c=c(\sigma)>0$ such that
        $
            \phi(r)\leq c\,r^\sigma
        $
        for all $r>0$. Combining this with
        $$
            \I_{[t,\infty)}\left(S(s)\right)\leq
            \frac{2S(s)}{t+S(s)},\qquad
            \frac{x}{1+x}=\int_0^\infty\left(1-\eup^{-xr}\right)
            \eup^{-r}\,\dup r,\quad x>0,
        $$
        and Tonelli's theorem, we get that for all $s,t>0$
        \begin{align*}
            \Pp\left(S(s)\geq t\right)
            &= \Ee\left[\I_{[t,\infty)}\left(S(s)\right)\right]\\
            &\leq 2\Ee\left[\frac{S(s)/t}{1+S(s)/t}\right]\\
            &= 2\Ee\left[\int_0^\infty\left(1-\eup^{-rS(s)/t}\right)\eup^{-r}\,\dup r\right]\\
            &= 2\int_0^\infty\left(1-\eup^{-s\phi(r/t)}\right)\eup^{-r}\,\dup r\\
            &\leq 2s\int_0^\infty\phi\left(\frac rt\right)\eup^{-r}\,\dup r\\
            &\leq 2c\,s\int_0^\infty\left(\frac rt\right)^\sigma\eup^{-r}\,\dup r\\
            &= 2c\,\Gamma(\sigma+1)st^{-\sigma}.
        \end{align*}
        This yields for all $t>0$
        \begin{align*}
            \Ee\big[\big(S^{-1}(t)\big)^{-\theta}\big]
            &= \int_0^\infty\Pp\big(S(s^{-1/\theta})\geq t\big)\,\dup s\\
            &\leq \int_0^\infty\left(1\wedge\left[2c\,\Gamma(\sigma+1)s^{-1/\theta}t^{-\sigma}\right]\right)\,\dup s\\
            &= \frac{1}{1-\theta}\left[2c\,\Gamma(\sigma+1)\right]^\theta t^{-\sigma\theta}.
        \qedhere
        \end{align*}
\end{proof}

\begin{remark}\label{sde-60}
    Let $\alpha\in (0,1)$. For an $\alpha$-stable
    subordinator $S$, the corresponding
    Bernstein function
    $\phi(r)=r^\alpha$ satisfies the conditions
    of Corollary \ref{sde-54} with $\sigma=\alpha$.
    Because of \eqref{sde-e59} and the well
    known two-sided estimate
    $$
        \Pp\left(S(r)\geq t\right)
        \asymp 1\wedge\left[rt^{-\alpha}
        \right],\quad r,t>0,
    $$
    ($f\asymp g$ means that $c^{-1}f(t)\leq g(t)\leq c f(t)$ for some $c\geq 1$ and all $t$) we have for any $t>0$
    \begin{equation}\label{sde-e63}
    \begin{aligned}
            \Ee\big[\big(S^{-1}(t)\big)^{-\theta}\big]
            &=\int_0^\infty\Pp\big(\big(S^{-1}(t)\big)^{-\theta}\geq s\big)\,\dup s\\
            &=\int_0^\infty\Pp\big(S(s^{-1/\theta})\geq t\big)\,\dup s\\
            &\asymp\int_0^\infty\left(1\wedge\left[s^{-1/\theta}t^{-\alpha}\right]\right)\,\dup s\\
            &=\frac{1}{1-\theta}\,t^{-\alpha\theta}.
    \end{aligned}
    \end{equation}
    This shows that Lemma \ref{sde-58} is sharp for $\alpha$-stable subordinators.
\end{remark}

\begin{remark}\label{sde-62}
    Let $Z$ be an inverse $\alpha$-stable subordinator, i.e.\ $Z(t)=S^{-1}(t)$ for $t\geq 0$, where $(S(t))_{t\geq 0}$ is an $\alpha$-stable subordinator. For any $t,\theta,\delta>0$ we have
    $$
        \Ee\exp\left[\frac{\delta}{Z(t)^\theta}\right]=\infty.
    $$
    The proof is similar to \eqref{sde-e63}:
    \begin{align*}
        \Ee\exp\left[\frac{\delta}{Z(t)^\theta}\right]
        &\geq \int_1^\infty\Pp\left(\exp\left[\delta \left(S^{-1}(t)\right)^{-\theta}\right]\geq s\right)\,\dup s\\
        &= \int_1^\infty\Pp\left(S\left(\left(\delta/\log s\right)^{1/\theta}\right)\geq t\right)\,\dup s\\
        &\asymp \int_1^\infty\left(1\wedge\left[\left(\delta/\log s\right)^{1/\theta}t^{-\alpha}\right]\right)\,\dup s\\
        &=\infty.
    \end{align*}

    This means that we cannot expect, in the setting of Corollary \ref{sde-54}, to get a power-Harnack inequality as we did in Corollary \ref{sde-52}~iii).
\end{remark}

\section{SDEs with non-Lipschitz drift and anisotropic noise}\label{non-lip}

Let $W^{H_i,(i)}=(W^{H_i,(i)}_t)_{t\geq 0}$, $Z^{(i)}=(Z^{(i)}(t))_{t\geq 0}$, $1\leq i\leq d$, $V=(V_t)_{t\geq 0}$ be $2d+1$ independent stochastic processes such that
\begin{equation}\label{nl-e03}
\left\{
\begin{aligned}
    &\text{$W^{H_i,(i)}$ are fBMs on $\real$ with Hurst parameter $H_i\in(0,1/2)$};\\
    &\text{$Z^{(i)}$ are non-decreasing processes on $[0,\infty)$ with $Z^{(i)}(0)=0$};\\
    &\text{$V$ is a locally bounded measurable
    process on $\real^d$ with $V_0=0$}.
\end{aligned}\right.
\end{equation}
We consider the following stochastic equation on $\real^d$:
\begin{equation}\label{nl-e05}
    X_t(x)
    = x+\int_0^tb_s\big(X_s(x)\big)\,\dup s + \big(W^{H_1,(1)}_{Z^{(1)}(t)},\dots, W^{H_d,(d)}_{Z^{(d)}(t)}\big) + V_t,\quad t\geq 0,\,x\in\real^d,
\end{equation}
where $b=(b^{(1)},\dots,b^{(d)}):[0,\infty)\times\real^d\to\real^d$, $b=b_t(x)$, is measurable, locally bounded in the variable $t\geq 0$ and continuous as a function of $x$. By $\mathscr{U}$ we denote the family of functions $u:(0,\infty)\to(0,\infty)$ which are continuous, non-decreasing, grow at most  linearly as $x\to\infty$ and satisfy $\int_{0+}\frac{\dup s}{u(s)}=\infty$. Typical examples of such functions are $u(s)=s$, $u(s)=s\log(\eup\vee s^{-1})$, $u(s)=s\cdot\{\log(\eup\vee s^{-1})\}\cdot \log\log(\eup^\eup\vee s^{-1})$.

In this section, we will use the $\ell^1$-norm on $\real^d$ which we denote by $\|x\|_1:= |x^{(1)}|+\dots+|x^{(d)}|$, $x\in\real^d$, and we replace the one-sided Lipschitz condition \eqref{intro-e03} by the
following Yamada--Watanabe-type condition
\begin{equation}\tag{A}\label{nl-e07}
\left\{
\begin{aligned}
    &\text{There exists some $u\in\mathscr{U}$ and a locally bounded}\\
    &\text{measurable function $k:[0,\infty)\to[0,\infty)$ such that}\\
    &\|b_t(x)-b_t(y)\|_1
        \leq k(t) u(\|x-y\|_1),
        \; t\geq 0,\: x,y\in\real^d.
\end{aligned}\right.
\end{equation}
As in Section 3, it is easy to see that \eqref{nl-e07} guarantees the existence, uniqueness and non-explosion of the
solution to \eqref{nl-e05}. We define for bounded Borel
functions $f:\real^d\to\real$ the operator
$$
    P_tf(x):=\Ee f(X_t(x))\,\quad
    t\geq 0,\;x\in\real^d.
$$
\begin{remark}
    Note that it is, in general, difficult to compare \eqref{nl-e07} with the condition \eqref{intro-e03} used in Section~\ref{sde}, since neither of them implies the other one.
\end{remark}

\subsection{Statement of the main result}

Let $k(t)$ be the constant appearing in \eqref{nl-e07}, and denote by $K(t)=\int_0^t k(s)\,\dup s$ its primitive. 
For $i\in\{1,\dots,d\}$ we define $\Theta_{H_i}$ by \eqref{sde-e13} with $H_i$ instead of $H$. Finally, we set for $u\in\mathscr{U}$ and $k(t)$
\begin{gather*}
    G_u(r):=
    \begin{cases}
        -\int_r^1\frac{\dup s}{u(s)},&\text{if\ \ }r\in[0,1),\\
        \int_1^r\frac{\dup s}{u(s)},&\text{if\ \ }r\in[1,\infty),
    \end{cases}
\end{gather*}
and
$$
    \Phi_{u,k}(t,r)
    := r + \int_0^t k(s)\,u\circ G_u^{-1} \big(G_u(r)+K(s)\big)\,\dup s,
    \quad t,r\geq 0;
$$
$G_u^{-1}$ is the inverse function of $G_u$. Since $u\in\mathscr{U}$, it is easy to see that $G_u$ is strictly increasing with $G_u(0)=-\infty$ and $\lim_{r\uparrow\infty}G_u(r)=\infty$, so that $\Phi_{u,k}$ is well-defined. If, in particular, $u(s)=cs$ for some constant $c>0$, then
$$
    \Phi_{u,k}(t,r)
    =\left(1+c\int_0^t k(s)\,\eup^{cK(s)}\,\dup s\right)r,
    \quad t,r\geq 0.
$$
Since we use the $\ell^1$-norm in this section, the local Lipschitz constant of a function $f$ on $\real^d$ at the point $x\in\real^d$ is defined by
$$
    |\nabla f|(x)
    :=\limsup_{y\to x}\frac{|f(y)-f(x)|}{\|y-x\|_1}.
$$

\begin{theorem}\label{nl-02}
    Assuming \eqref{nl-e03} and \eqref{nl-e07},
    let $X_t(x)$ denote the unique solution
    to the SDE \eqref{nl-e05}.

    \smallskip\noindent\textup{i)}
        For $T>0$, $x,y\in\real^d$, and any bounded Borel function $f:\real^d\to [1,\infty)$
        $$
            P_T\log f(y)
            \leq \log P_Tf(x) + \Phi_{u,k}^2\left(T,\|x-y\|_1\right) \sum_{i=1}^d \Ee\left[\frac{\Theta_{H_i}}{(Z^{(i)}(T))^{2H_i}}\right].
        $$

    \smallskip\noindent\textup{ii)}
        For $T>0$, $x,y\in\real^d$, $p>1$ and any bounded Borel function $f:\real^d\to [0,\infty)$
        $$
            \big(P_Tf(y)\big)^p
            \leq
            P_Tf^p(x)\cdot
            \left(\Ee\exp\left[\frac{p}{(p-1)^2} \Phi_{u,k}^2\left(T,\|x-y\|_1\right) \sum_{i=1}^d\frac{\Theta_{H_i}}{(Z^{(i)}(T))^{2H_i}}\right]\right)^{p-1}.
        $$

    \smallskip\noindent\textup{iii)}
        If \eqref{nl-e07} holds with $u(s)=cs$ for some constant $c>0$, then we have for $T>0$, $x\in\real^d$ and any bounded Borel function $f:\real^d\to\real$
        \begin{align*}
            |\nabla P_Tf|^2(x)
            &\leq\left\{P_Tf^2(x)-\big(P_Tf(x)\big)^2\right\}\\
            &\qquad\qquad\mbox{}\times 2 \left(1+c\int_0^T k(s)\,\eup^{cK(s)}\,\dup s\right)^2
            \sum_{i=1}^d\Ee\left[\frac{\Theta_{H_i}}{(Z^{(i)}(T))^{2H_i}}\right].
        \end{align*}
\end{theorem}

\subsection{Deterministic time-changes}

The proof of Theorem~\ref{nl-02} uses the same strategy as the proof of Theorem~\ref{sde-02}. Because of the independence of the random time-change and the driving processes, we consider first a deterministic time-change $\ell=(\ell^{(1)},\dots,\ell^{(d)}):[0,\infty) \to [0,\infty)^d$ such that
for each $i\in\{1,\dots,d\}$ the map $t\mapsto\ell^{(i)}(t)$ is non-decreasing and c\`{a}dl\`{a}g with $\ell^{(i)}(0)=0$. Let $v=(v^{(1)},\dots,v^{(d)}):[0,\infty)\to\real^d$ be a locally bounded measurable function such that $v_0=0$. Under \eqref{nl-e07}, the following
SDE has a unique non-explosive strong solution
\begin{equation}\label{nl-e11}
    X_t^{\ell,v}(x)=x+\int_0^tb_s\big(X_s^{\ell,v}(x)\big)\,\dup s
    +\big(W^{H_1,(1)}_{\ell^{(1)}(t)},\dots,
    W^{H_d,(d)}_{\ell^{(d)}(t)}\big)
    +v_t,\quad t\geq 0,\,x\in\real^d.
\end{equation}
As before, we set for any bounded Borel function $f:\real^d\to\real$
$$
    P_t^{\ell,v}f(x)
    := \Ee f\big(X_t^{\ell,v}(x)\big)\,\quad
    t\geq 0,\; x\in\real^d.
$$

\begin{proposition}\label{nl-04}
    Assuming \eqref{nl-e03} and \eqref{nl-e07}, let $X^{\ell,v}(x)$ denote the unique solution to the SDE \eqref{nl-e11}.

    \smallskip\noindent\textup{i)}
        For $T>0$, $x,y\in\real^d$ and any bounded Borel function $f:\real^d\to[1,\infty)$
        $$
            P_T^{\ell,v}\log f(y)
            \leq \log P_T^{\ell,v}f(x) + \Phi_{u,k}^2\left(T,\|x-y\|_1\right) \sum_{i=1}^d \frac{\Theta_{H_i}}{\left(\ell^{(i)}(T)\right)^{2H_i}}.
        $$

    \smallskip\noindent\textup{ii)}
        For $T>0$, $p>1$, $x,y\in\real^d$ and any bounded Borel function $f:\real^d\to[0,\infty)$
        $$
            \big(P_T^{\ell,v}f(y)\big)^p
            \leq P_T^{\ell,v}f^p(x)\cdot
            \exp\left[\frac{p}{p-1} \Phi_{u,k}^2\left(T,\|x-y\|_1\right)\sum_{i=1}^d \frac{\Theta_{H_i}}{\left(\ell^{(i)}(T)\right)^{2H_i}}\right].
        $$
\end{proposition}

As in Section \ref{sde}.2, we approximate $\ell^{(i)}$ by strictly increasing, absolutely continuous functions
$$
    \ell_\epsilon^{(i)}(t)
    :=
    \frac{1}{\epsilon} \int_t^{t+\epsilon}\ell^{(i)}(s)\,\dup s + \epsilon t
    =
    \int_0^1\ell^{(i)}(\epsilon s+t)\,\dup s + \epsilon t,
    \quad
    \epsilon\in(0,1),\; 1\leq i\leq d,\; t\geq 0.
$$
By construction, $\ell_\epsilon^{(i)}(t)\downarrow\ell^{(i)}(t)$ as $\epsilon\downarrow0$. Denote by $\gamma_\epsilon^{(i)}$ the inverse function of $\ell_\epsilon^{(i)}$. We consider the following approximation of the SDE \eqref{nl-e11}
\begin{equation}\label{nl-e13}
    X_t^{\ell_\epsilon,v,(i)}(x)
    =
    x^{(i)} + \int_0^tb_s^{(i)}\big(X_s^{\ell_\epsilon,v}(x)\big)\,\dup s + W^{H_i,(i)}_{\ell_\epsilon^{(i)}(t)-\ell_\epsilon^{(i)}(0)} + v^{(i)}_t,
    \quad i=1,\dots, d,
\end{equation}
where $t\geq 0$, $x\in\real^d$ and $X_t^{\ell_\epsilon,v}(x) = \big(X_t^{\ell_\epsilon,v,(1)}(x),\dots,X_t^{\ell_\epsilon,v,(d)}(x)\big)$. Again, for all bounded Borel functions $f:\real^d\to\real$
$$
    P_t^{\ell_\epsilon,v}f(x)
    :=
    \Ee f\big(X_t^{\ell_\epsilon,v}(x)\big),
    \quad t\geq 0,\; x\in\real^d.
$$
We will first prove the Harnack inequalities for $P_t^{\ell_\epsilon,v}$ using a modification of the arguments from Lemma \ref{sde-12}, compare also \cite{WZ15}.

\begin{lemma}\label{nl-06}
    Assuming \eqref{nl-e03} and \eqref{nl-e07}, let $X_t^{\ell_\epsilon,v}(x)$, $\epsilon\in(0,1)$, denote the unique solution to the SDE \eqref{nl-e13}.

    \smallskip\noindent\textup{i)}
        For $T>0$, $x,y\in\real^d$ and any bounded Borel function $f:\real^d\to [1,\infty)$
        $$
            P_T^{\ell_\epsilon,v}\log f(y)
            \leq \log P_T^{\ell_\epsilon,v}f(x) + \Phi_{u,k}^2\left(T,\|x-y\|_1\right)
            \sum_{i=1}^d \frac{\Theta_{H_i}}{\left(\ell_\epsilon^{(i)}(T)-\ell_\epsilon^{(i)}(0)\right)^{2H_i}}.
        $$

    \smallskip\noindent\textup{ii)}
        For $T>0$, $p>1$, $x,y\in\real^d$ and any bounded Borel function $f:\real^d\to [0,\infty)$
        $$
            \big(P_T^{\ell_\epsilon,v}f(y)\big)^p
            \leq
            P_T^{\ell_\epsilon,v}f^p(x)\cdot
            \exp\left[\frac{p}{p-1} \Phi_{u,k}^2\left(T,\|x-y\|_1\right)
            \sum_{i=1}^d \frac{\Theta_{H_i}}{\left(\ell_\epsilon^{(i)}(T) - \ell_\epsilon^{(i)}(0)\right)^{2H_i}}\right].
        $$
\end{lemma}

\begin{proof}
    Fix $\epsilon\in (0,1)$, $T>0$, $x,y\in\real^d$ and denote the coordinates by a superscript. Let $(Y_t)_{t\geq 0}$ be a solution of the equation
    \begin{equation}\label{nl-e15}
    \begin{aligned}
        Y_t^{(i)}
        =y^{(i)} &+ \int_0^tb_s^{(i)}(Y_s)\,\dup s + W^{H_i,(i)}_{\ell_\epsilon^{(i)}(t)-\ell_\epsilon^{(i)}(0)} + v^{(i)}_t\\
        &\mbox{}+\xi^{(i)}\int_0^t\I_{[0,\tau_i)}(s) \frac{X_s^{\ell_\epsilon,v,(i)}(x)-Y^{(i)}_s}{\big|X_s^{\ell_\epsilon,v,(i)}(x)-Y^{(i)}_s\big|}\,
        \dup\ell_\epsilon^{(i)}(s),
    \end{aligned}
    \end{equation}
    where $t\geq 0$, $i=1,\dots, d$ and
    $$
        \xi^{(i)}
        :=\frac{\Phi_{u,k}(T,\|x-y\|_1)}{\ell_\epsilon^{(i)}(\delta T)-\ell_\epsilon^{(i)}(0)},
        \quad
        \delta\in(0,1),
        \quad
        \tau_i:=\inf\big\{t\geq 0\,:\, X_t^{\ell_\epsilon,v,(i)}(x)=Y^{(i)}_t\big\}.
    $$
    As in the proof of Lemma~\ref{sde-12}, there is a unique solution to \eqref{nl-e15} such that $Y^{(i)}_t=X_t^{\ell_\epsilon,v,(i)}(x)$ for $t\geq\tau_i$ on the set $\{\tau_i < \infty\}$, and we use
    the differential versions of the equations
    \eqref{nl-e13} and \eqref{nl-e15} along
    with the observation that $\dup|\zeta_t|
    =\I_{\{\zeta_t\neq0\}}|\zeta_t|^{-1}\zeta_t\,
    \dup\zeta_t$, where $\zeta_t:=X_t^{\ell_\epsilon,v,(i)}(x)
    -Y^{(i)}_t$, to get for $i=1,\dots,d$ and $t\geq 0$
    \begin{align*}
        &\big|X_t^{\ell_\epsilon,v,(i)}(x)-Y^{(i)}_t\big| -|x^{(i)}-y^{(i)}|\\
        &\;=\int_0^{t\wedge\tau_i} \frac{X_s^{\ell_\epsilon,v,(i)}(x)-Y^{(i)}_s}{\big|X_s^{\ell_\epsilon,v,(i)}(x)-Y^{(i)}_s\big|}
        \big(b^{(i)}_s\big(X_s^{\ell_\epsilon,v}(x)\big)-b^{(i)}_s(Y_s)\big)\,\dup s
        -\xi^{(i)}\big[\ell_\epsilon^{(i)}(t\wedge\tau_i)-\ell_\epsilon^{(i)}(0)\big]\\
        &\;\leq\int_0^{t\wedge\tau_i} \big|b^{(i)}_s\big(X_s^{\ell_\epsilon,v}(x)\big)-b^{(i)}_s(Y_s)\big|\,\dup s
        -\xi^{(i)}\big[\ell_\epsilon^{(i)}(t\wedge\tau_i) - \ell_\epsilon^{(i)}(0)\big]\\
        &\;\leq\int_0^t \big|b^{(i)}_s\big(X_s^{\ell_\epsilon,v}(x)\big)
        -b^{(i)}_s(Y_s)\big|\,\dup s
        -\xi^{(i)}\big[\ell_\epsilon^{(i)}(t\wedge\tau_i) - \ell_\epsilon^{(i)}(0)\big].
    \end{align*}
    Summing over $i$ we obtain, using \eqref{nl-e07},
    \begin{align*}
        &\big\|X_t^{\ell_\epsilon,v}(x)-Y_t\big\|_1\\
        &\qquad\leq
        \|x-y\|_1 + \int_0^t\big\|b_s\big(X_s^{\ell_\epsilon,v}(x)\big) - b_s(Y_s)\big\|_1\,\dup s
        - \sum_{i=1}^d\xi^{(i)} \big[\ell_\epsilon^{(i)}(t\wedge\tau_i)-\ell_\epsilon^{(i)}(0)\big]\\
        &\qquad\leq
        \|x-y\|_1 + \int_0^t k(s)\, u\big(\big\|X_s^{\ell_\epsilon,v}(x) - Y_s\big\|_1\big)\,\dup s
        -\sum_{i=1}^d\xi^{(i)}\big[\ell_\epsilon^{(i)}(t\wedge\tau_i) - \ell_\epsilon^{(i)}(0)\big].
    \end{align*}
    We can now apply Bihari's inequality (cf.\ \cite[Section 3]{Bi})
    to conclude
    $$
        \big\|X_s^{\ell_\epsilon,v}(x)-Y_s\big\|_1\leq
        G_u^{-1}\big(G_u(\|x-y\|_1)+K(s)\big),
        \quad s\geq 0.
    $$
    Inserting this into the previous inequality yields for any $t\in[0,T]$
    \begin{align*}
        \sum_{i=1}^d &\xi^{(i)} \big[\ell_\epsilon^{(i)}(t\wedge\tau_i)-\ell_\epsilon^{(i)}(0)\big]\\
        &\leq
            \|x-y\|_1 + \int_0^t k(s)\,u\circ G_u^{-1}\big(G_u(\|x-y\|_1) + K(s)\big)\,\dup s\\
        &\leq
            \Phi_{u,k}(T,\|x-y\|_1),
    \end{align*}
    which means that we have for each $n=1,\dots,d$
    $$
        \frac{\ell_\epsilon^{(n)}(t\wedge\tau_n)-\ell_\epsilon^{(n)}(0)}{\ell_\epsilon^{(n)}(\delta T)-\ell_\epsilon^{(n)}(0)}
        \leq
        \sum_{i=1}^d
        \frac{\ell_\epsilon^{(i)}(t\wedge\tau_i)-\ell_\epsilon^{(i)}(0)}{\ell_\epsilon^{(i)}(\delta T)-\ell_\epsilon^{(i)}(0)}
        \leq 1.
    $$
    Taking $t=T$ implies $\ell_\epsilon^{(n)}(T\wedge\tau_n) \leq \ell_\epsilon^{(n)}(\delta T)$ and this is only possible if $\tau_n < T$ as $\delta\in (0,1)$ and $\ell_\epsilon^{(n)}$ is strictly increasing.

    Let
    $$
        g^{(i)}_r
        :=
        \xi^{(i)} \I_{[0,\ell_\epsilon^{(i)}(\tau_i)-\ell_\epsilon^{(i)}(0))}(r) \frac{X_{\gamma_\epsilon^{(i)}(r+\ell_\epsilon^{(i)}(0))}^{\ell_\epsilon,v,(i)}(x)-Y^{(i)}_{\gamma_\epsilon^{(i)}(r+\ell_\epsilon^{(i)}(0))}}
        {\big|X_{\gamma_\epsilon^{(i)}(r+\ell_\epsilon^{(i)}(0))}^{\ell_\epsilon,v,(i)}(x)-Y^{(i)}_{\gamma_\epsilon^{(i)}(r+\ell_\epsilon^{(i)}(0))}\big|},
        \quad r\geq 0,\,1\leq i\leq d.
    $$
    Then by \eqref{jg5fhnhg}, we know that
    $\int_0^\bullet g_r^{(i)}\,\dup r\in I_{0+}^{H+1/2}(L^2[0,\ell_\epsilon^{(i)}(\tau_i)
    -\ell_\epsilon^{(i)}(0)])$. Let $W^{(i)}=(W^{(i)})_{t\geq 0}$, $1\leq i\leq d$, be independent one-dimensional standard Brownian
    motions, and define
    $$
        \eta_s^{(i)}
        :=
        \I_{[0,\tau_i)}(s) \mathcal{K}_{H_i}^{-1}\left(\int_0^\bullet
         g_r^{(i)}\,\dup r\right) \left(\ell_\epsilon^{(i)}(s)- \ell_\epsilon^{(i)}(0)\right),\quad s\geq 0,
         \,1\leq i\leq d,
    $$
    \begin{align*}
        M_t
        :=& -\sum_{i=1}^d \int_0^t\I_{[0,\tau_i)}(s) \eta_s^{(i)}\,\dup W^{(i)}_{\ell_\epsilon^{(i)}(s)-\ell_\epsilon^{(i)}(0)}\\
        =& -\sum_{i=1}^d\int_0^{\ell_\epsilon^{(i)}(t\wedge\tau_i)-\ell_\epsilon^{(i)}(0)}\eta_{\gamma_\epsilon^{(i)}(s+\ell_\epsilon^{(i)}(0))}^{(i)}\,\dup W^{(i)}_s,\quad t\geq0,
    \end{align*}
    $$
        \widetilde{W}^{(i)}_{\ell_\epsilon^{(i)}(t)-\ell_\epsilon^{(i)}(0)}
        :=
        W^{(i)}_{\ell_\epsilon^{(i)}(t)-\ell_\epsilon^{(i)}(0)}+\int_0^t\eta_s^{(i)}\,\dup\ell_\epsilon^{(i)}(s),
        \quad t\geq 0,\;1\leq i\leq d.
    $$
    Noting $H_i\in(0,1/2)$ and using \eqref{bas-e03}
    we find for $s\in[0,\ell_\epsilon^{(i)}(\tau_i)-\ell_\epsilon^{(i)}(0))$
    \begin{align*}
        \big|\eta^{(i)}_{\gamma_\epsilon^{(i)}(s+\ell_\epsilon^{(i)}(0))}\big|
        &= \left|\frac{1}{\Gamma\big(\frac12-H_i\big)} s^{H_i-\frac12}\int_0^sr^{\frac12-H_i}g_r^{(i)}(s-r)^{-H_i-\frac12}\, \dup r\right|\\
        &\leq \frac{1}{\Gamma\big(\frac12-H_i\big)}\,\xi^{(i)}s^{H_i-\frac12}\int_0^s r^{\frac12-H_i}(s-r)^{-H_i-\frac12}\,\dup r\\
        &= \frac{B\big(\frac32-H_i,\frac12-H_i\big)}{\Gamma\big(\frac12-H_i\big)}\,\xi^{(i)} s^{\frac12-H_i}.
    \end{align*}
    Therefore, the compensator of the
    martingale $M$ satisfies
    \begin{align*}
        \langle M\rangle_\infty
        &= \sum_{i=1}^d\int_0^{\ell_\epsilon^{(i)}(\tau_i)-\ell_\epsilon^{(i)}(0)}\big|\eta^{(i)}_{\gamma_\epsilon^{(i)}(s+\ell_\epsilon^{(i)}(0))}\big|^2\,\dup s\\
        &\leq\sum_{i=1}^d \left(\frac{B\left(\frac32-H_i,\frac12-H_i\right)}{\Gamma\left(\frac12-H_i\right)}\right)^2 (\xi^{(i)})^2
        \int_0^{\ell_\epsilon^{(i)}(T)-\ell_\epsilon^{(i)}(0)} s^{1-2H_i}\,\dup s\\
        &= 2\Phi_{u,k}^2\left(T,\|x-y\|_1\right) \sum_{i=1}^d\Theta_{H_i}
        \frac{\left[\ell_\epsilon^{(i)}(T)-\ell_\epsilon^{(i)}(0)\right]^{2-2H_i}}{\left[\ell_\epsilon^{(i)}(\delta T)-\ell_\epsilon^{(i)}(0)\right]^2}.
    \end{align*}
    Set
    $$
        R:=\exp\left[M_\infty-\frac12\langle M\rangle_\infty
        \right].
    $$
    Since $\Ee\,\eup^{\frac12\langle M\rangle_\infty}<
    \infty$, we can use Novikov's criterion to obtain $\Ee R=1$, and by Girsanov's theorem we get
    $$
        \widetilde{W}_{\ell_\epsilon(t)-\ell_\epsilon(0)}
        :=\big(\widetilde{W}^{(1)}
        _{\ell_{\epsilon}^{(1)}(t)-\ell_{\epsilon}^{(1)}(0)},
        \dots,
        \widetilde{W}^{(d)}
        _{\ell_{\epsilon}^{(d)}(t)-\ell_{\epsilon}^{(d)}(0)}
        \big),\quad t\geq 0,
    $$
    is a $d$-dimensional $(\mathscr{F}^{\ell_\epsilon}_t)$-martingale under $R\Pp$, where $\mathscr{F}^{\ell_\epsilon}_t$ is the $\sigma$-algebra generated by $\{W^{(i)}_{\ell_\epsilon^{(i)}(s)-\ell_\epsilon^{(i)}(0)}\,:\,0\leq s\leq t,\, 1\leq i\leq d\}$. For $0\leq s\leq t$ and
    $\theta=(\theta^{(1)},\dots,\theta^{(d)})\in\real^d$, it is easy
    to see that
    \begin{align*}
        \Ee_{R\Pp}\big(
        &\exp\big[\iup \big\langle\theta, \widetilde{W}_{\ell_\epsilon(t)-\ell_\epsilon(0)} -\widetilde{W}_{\ell_\epsilon(s)-\ell_\epsilon(0)}
        \big\rangle\big]\,\big|\,\mathscr{F}^{\ell_\epsilon}_s\big)\\
        &= \exp\left[\frac12 \sum_{i=1}^d (\theta^{(i)})^2\left[\ell_\epsilon^{(i)}(t)-\ell_\epsilon^{(i)}(s)\right]\right]\\
        &= \Ee\big(\exp\big[\iup \big\langle\theta, W_{\ell_\epsilon(t)-\ell_\epsilon(0)}-W_{\ell_\epsilon(s)-\ell_\epsilon(0)} \big\rangle\big]\,\big|\,\mathscr{F}^{\ell_\epsilon}_s\big),
    \end{align*}
    where
    $$
        W_{\ell_\epsilon(t)-\ell_\epsilon(0)}
        :=\big(W^{(1)}_
        {\ell_{\epsilon}^{(1)}(t)
        -\ell_{\epsilon}^{(1)}(0)},
        \dots,W^{(d)}_{\ell_{\epsilon}^{(d)}(t)
        -\ell_{\epsilon}^{(d)}(0)}
        \big),\quad t\geq 0.
    $$
    This shows that the distribution of $(\widetilde{W}_{\ell_\epsilon(t)-\ell_\epsilon(0)})_{t\geq 0}$ under $R\Pp$ coincides with the law of $(W_{\ell_\epsilon(t)-\ell_\epsilon(0)})_{t\geq 0}$ under $\Pp$. If we rewrite \eqref{nl-e13} and \eqref{nl-e15} for $t\geq 0$ and $i=1,\dots, d$ as
    \begin{align*}
        X_t^{\ell_\epsilon,v,(i)}(x)
        = x^{(i)} &+ \int_0^tb_s^{(i)}\big(X_s^{\ell_\epsilon,v}(x)\big)\,\dup s+v_t^{(i)}\\
        &\mbox{}+\int_0^t \mathcal{K}_{H_i}\left(\ell_\epsilon^{(i)}(t)-\ell_\epsilon^{(i)}(0), \ell_\epsilon^{(i)}(s)-\ell_\epsilon^{(i)}(0)\right)
        \dup W^{(i)}_{\ell_\epsilon^{(i)}(s)-\ell_\epsilon^{(i)}(0)},
    \end{align*}
    and
    \begin{align*}
        Y_t^{(i)}
        = y^{(i)} &+ \int_0^tb_s^{(i)}(Y_s)\,\dup s+v_t^{(i)}\\
        &\mbox{}+\int_0^t \mathcal{K}_{H_i}\left(\ell_\epsilon^{(i)}(t)-\ell_\epsilon^{(i)}(0), \ell_\epsilon^{(i)}(s)-\ell_\epsilon^{(i)}(0)\right)
        \dup \widetilde{W}^{(i)}_{\ell_\epsilon^{(i)}(s)-\ell_\epsilon^{(i)}(0)},
    \end{align*}
    respectively, we see that the distribution of $(X_t^{\ell_\epsilon,v}(y))_{t\geq 0}$ under $\Pp$ coincides with the distribution of $(Y_t)_{t\geq 0}$ under $R\Pp$.

    As in the proof of Lemma \ref{sde-12} we get for any bounded Borel function $f:\real^d\to [1,\infty)$
    \begin{align*}
        P_T^{\ell_\epsilon,v}\log f(y)
        &= \Ee\big[R\log f\big(X_T^{\ell_\epsilon,v}(x)\big)\big]\\
        &\leq \log\Ee f\big(X_T^{\ell_\epsilon,v}(x)\big) + \Ee\left[R\log R\right]\\
        &= \log P_T^{\ell_\epsilon,v}f(x) + \Ee_{R\Pp}\log R\\
        &\leq \log P_T^{\ell_\epsilon,v}f(x) + \Phi_{u,k}^2\left(T,\|x-y\|_1\right) \sum_{i=1}^d\Theta_{H_i}
        \frac{\left[\ell_\epsilon^{(i)}(T)-\ell_\epsilon^{(i)}(0)\right]^{2-2H_i}}{\left[\ell_\epsilon^{(i)}(\delta T)-\ell_\epsilon^{(i)}(0)\right]^2},
    \end{align*}
    and for any non-negative bounded Borel function $f:\real^d\to[0,\infty)$
    \begin{align*}
        \big(&P_T^{\ell_\epsilon,v}f(y)\big)^p\\
        &= \big(\Ee\big[Rf\big(X_T^{\ell_\epsilon,v}(x)\big)\big]\big)^p\\
        &\leq \big(\Ee f^p\big(X_T^{\ell_\epsilon,v}(x)\big)\big) \left(\Ee\left[R^{p/(p-1)}\right]\right)^{p-1}\\
        &\leq P_T^{\ell_\epsilon,v}f^p(x)\cdot
        \exp\left[
            \frac{p}{p-1}\Phi_{u,k}^2\left(T,\|x-y\|_1\right)
            \sum_{i=1}^d\Theta_{H_i}
            \frac{\left[\ell_\epsilon^{(i)}(T)-\ell_\epsilon^{(i)}(0)\right]^{2-2H_i}}{\left[\ell_\epsilon^{(i)}(\delta T)-\ell_\epsilon^{(i)}(0)\right]^2}
        \right].
    \end{align*}
    Letting $\delta\uparrow 1$ finishes the proof.
\end{proof}

The following result is easy; for the sake of completeness, we include its simple proof.

\begin{lemma}\label{gf499d}
    Assume \eqref{nl-e07}. Then for
    any $x\in\real^d$ and $t\geq 0$,
    $$
        \lim_{\epsilon\downarrow 0}
        X_t^{\ell_\epsilon,v}(x)=
        X_t^{\ell,v}(x).
    $$
\end{lemma}

\begin{proof}
Fix $T>0$, $\epsilon\in (0,1)$, $x\in\real^d$ and
observe that for $t\in[0,T]$
\begin{align*}
    &\big\|X_t^{\ell_\epsilon,v}(x)-X_t^{\ell,v}(x)\big\|_1\\
    &\leq \int_0^t\big\|b_s\big(X_s^{\ell_\epsilon,v}(x)\big)
    -b_s\big(X_s^{\ell,v}(x)\big)\big\|_1\,\dup s
        +\sum_{i=1}^d\big|W^{H_i,(i)}_{\ell_\epsilon^{(i)}(t)
        -\ell_\epsilon^{(i)}(0)}-W^{H_i,(i)}_{\ell^{(i)}(t)}\big|\\
    &\leq\int_0^t k(s)\,u\big(\big\|X_s^{\ell_\epsilon,v}(x)-X_s^{\ell,v}(x)\big\|_1\big)\,\dup s
        +\sum_{i=1}^d\big|W^{H_i,(i)}_{\ell_\epsilon^{(i)}(t)
        -\ell_\epsilon^{(i)}(0)}-W^{H_i,(i)}_{\ell^{
        (i)}(t)}\big|.
\end{align*}
Since the processes $X^{\ell_\epsilon,v}_t(x)$ and $X^{\ell,v}_t(x)$ are non-explosive, the last integral expression is finite.
Applying Bihari's lemma with $g(\epsilon,t):=\sum_{i=1}^d\big|W^{H_i,(i)}_{\ell_\epsilon^{(i)}(t)
-\ell_\epsilon^{(i)}(0)}-W^{H_i,(i)}_{\ell^{(i)}(t)}\big|$ yields
that for any $t\in[0,T]$
$$
    \big\|X_t^{\ell_\epsilon,v}(x)-X_t^{\ell,v}(x)\big\|_1
    \leq G_u^{-1}\big(G_u(g(\epsilon,t))+K(t)\big).
$$
Since $\ell_\epsilon^{(i)}(t)\to\ell^{(i)}(t)$, one
has $g(\epsilon,t)\to 0$ as $\epsilon\downarrow 0$.
Combining this with $G_u(0+)=-\infty$, we find
$$
    \lim_{\epsilon\downarrow 0}\big\|X_t^{\ell_\epsilon,v}(x)
    -X_t^{\ell,v}(x)\big\|_1 = 0.
$$
Hence,
$$
    \lim_{\epsilon\downarrow0} X_t^{\ell_\epsilon,v}(x) = X_t^{\ell,v}(x)
    \quad \text{holds for all $t\in[0,T]$}.
$$
The claim follows since $T>0$ is arbitrary.
\end{proof}

\subsection{Proof of Theorem \ref{nl-02}}

    The proof parallels the argument which we have used for Theorem~\ref{sde-02}; in particular, Lemma~\ref{nl-06} plays now the same role as Lemma~\ref{sde-12} for the proof of Theorem~\ref{sde-02}.

    The first step is to prove the log- and power-Harnack inequalities stated in i) and ii) for deterministic time-changes and for continuous functions $f\in C_b(\real^d)$. Lemma~\ref{sde-02} has these inequalities for absolutely continuous time-changes and the operators $P^{\ell_\epsilon,v}$; letting $\epsilon\downarrow 0$, we get them for general time-changes and the operators $P^{\ell,v}$.

    Since the processes $Z$ and $V$ are independent of $(W^{H_1,(1)},\dots,W^{H_d,(d)})$, we can indeed treat them like deterministic processes $Z=\ell$ and $V=v$, i.e.\ just as in Theorem~\ref{sde-02} the deterministically time-changed inequalities combined with the Jensen and H\"{o}lder inequality prove Theorem~\ref{nl-02} i) and ii).

    Finally, the gradient estimate follows immediately from i) and \cite[Proposition 2.3]{ATW14}.

\subsection{Two examples}

As in Section \ref{sde}.3, we apply our results to two typical examples of stochastic time-changes $Z^{(i)}$: subordinators and inverse subordinators.

Throughout this section we assume that $(X_t(x))_{t\geq 0}$ is the unique non-explosive solution to the SDE \eqref{nl-e05} and $P_tf(x)=\Ee f(X_t(x))$. Combining Theorem \ref{nl-02} and \cite[Theorem 3.8\,(a) and (b)]{DS14}, we obtain the following result.

\begin{corollary}\label{nl-10}
    Assume that \eqref{nl-e03} and \eqref{nl-e07} hold, and that for each $i=1,\dots, d$, $Z^{(i)}$ is a subordinator with Bernstein function $\phi_i$ such that
    $\liminf_{r\to\infty}\phi_i(r)r^{-\alpha_i}>0$ for
    some $\alpha_i>0$. Let
    $$
        \kappa_1:=2\min_{1\leq i\leq d}\frac{H_i}{\alpha_i}
    \quad\text{and}\quad
        \kappa_2:=2\max_{1\leq i\leq d}\frac{H_i}{\alpha_i}.
    $$
    Then there exists some constant $C=C_{\alpha_1,\dots,\alpha_d, H_1,\dots,H_d}>0$ such that the following assertions \textup{i)--iii)} hold.

    \smallskip\noindent\textup{i)}
        For $T>0$, $x,y\in\real^d$ and all bounded Borel functions $f:\real^d\to [1,\infty)$
        $$
            P_T\log f(y)
            \leq
            \log P_Tf(x) + \frac{Cd}{(T\wedge1)^{\kappa_2}}\,\Phi_{u,k}^2\left(T,\|x-y\|_1\right).
        $$
     If, in addition, $\liminf_{r\downarrow 0} \phi_i(r)r^{-\alpha_i}>0$
     for each $i$, then
        $$
            P_T\log f(y)
            \leq
            \log P_Tf(x) + \frac{Cd}{T^{\kappa_1}\wedge T^{\kappa_2}}\,\Phi_{u,k}^2\left(T,\|x-y\|_1\right).
        $$

    \smallskip\noindent\textup{ii)}
        Assume that $\alpha_i>2H_i/(1+2H_i)$ for each $i=1,\dots,d$. For any $T>0$, $x,y\in\real^d$, $p>1$
        and all bounded Borel functions $f:\real^d\to [0,\infty)$
         \begin{align*}
            \big(P_Tf(y)\big)^p
            \leq &P_Tf^p(x)\cdot
                \exp\Bigg[\frac{Cp}{p-1}\,\Phi_{u,k}^2\left(T,\|x-y\|_1\right)\left(1+\frac{d}{T^{\kappa_1}\wedge T^{\kappa_2}}\right)\\
            &\quad\mbox{} + C(p-1)\sum_{i=1}^d\left(\frac{p\Phi_{u,k}^2\left(T,\|x-y\|_1\right)}{(p-1)^2}\right)^{\frac{\alpha_i}{\alpha_i-2H_i(1-\alpha_i)}}
            T^{-\frac{2H_i}{\alpha_i-2H_i(1-\alpha_i)}}\Bigg].
        \end{align*}
        If, in addition, $\liminf_{r\downarrow 0}\phi_i(r)r^{-\alpha_i}>0$
        for each $i$,
        then
        \begin{align*}
            \big(P_Tf(y)\big)^p
            \leq &P_Tf^p(x)\cdot
                \exp\Bigg[\frac{Cdp}{p-1}\frac{\Phi_{u,k}^2\left(T,\|x-y\|_1\right)}{T^{\kappa_1}\wedge T^{\kappa_2}}\\
                &\quad\mbox{}+C(p-1)\sum_{i=1}^d\left(\frac{p\Phi_{u,k}^2\left(T,\|x-y\|_1\right)}{(p-1)^2}\right)^{\frac{\alpha_i}{\alpha_i-2H_i(1-\alpha_i)}}
                    T^{-\frac{2H_i}{\alpha_i-2H_i(1-\alpha_i)}}\Bigg].
        \end{align*}

    \smallskip\noindent\textup{iii)}
        If \eqref{nl-e07} holds for $u(s)=cs$ and some constant $c>0$, then for $T>0$, $x\in\real^d$ and all bounded Borel functions $f:\real^d\to\real$
        $$
            |\nabla P_Tf|^2(x)
            \leq
            \left\{P_Tf^2(x)-\big(P_Tf(x)\big)^2\right\}
            \frac{Cd}{(T\wedge1)^{\kappa_2}}
            \left(1+c\int_0^T k(s)\,\eup^{cK(s)}\,\dup s\right)^2.
        $$
        If, in addition,
        $\liminf_{r\downarrow 0}\phi_i(r) r^{-\alpha_i}>0$
        for each $i$, then
        $$
            |\nabla P_Tf|^2(x)
            \leq
            \left\{P_Tf^2(x)-\big(P_Tf(x)\big)^2\right\}
            \frac{Cd}{T^{\kappa_1}\wedge T^{\kappa_2}}
            \left(1+c\int_0^T k(s)\,\eup^{cK(s)}\,\dup s\right)^2.
        $$
\end{corollary}

If the $Z^{(i)}$ are inverse subordinators, we cannot expect that a power-Harnack inequality will hold, see Remark \ref{sde-62}.
Combining Lemma \ref{sde-58} with Theorem \ref{nl-02} i) \& iii), we still have the following corollary.

\begin{corollary}\label{nl-12}
    Assume that \eqref{nl-e03} and \eqref{nl-e07} hold,
    and that $Z^{(i)}$ is for each $i=1,\dots,d$ an inverse subordinator with Bernstein function $\phi_i$.

    Moreover, assume that $\limsup_{r\downarrow 0}\phi_i(r)r^{-\alpha_i}<\infty$ and $\limsup_{r\to\infty}\phi_i(r)r^{-\alpha_i}<\infty$ for some $\alpha_i>0$. Let
    $$
        \kappa_3:=2\min_{1\leq i\leq d}H_i\alpha_i
    \quad\text{and}\quad
        \kappa_4:=2\max_{1\leq i\leq d}H_i\alpha_i.
    $$
    Then there exists some constant $C=C_{\alpha_1,\dots,\alpha_d, H_1,\dots,H_d}>0$ such that the following assertions \textup{i)}, \textup{ii)} hold.

    \smallskip\noindent\textup{i)}
        For $T>0$, $x,y\in\real^d$ and all bounded Borel functions $f: \real^d\to [1,\infty)$
        $$
            P_T\log f(y)
            \leq
            \log P_Tf(x) + \frac{Cd}{T^{\kappa_3}\wedge T^{\kappa_4}}\,\Phi_{u,k}^2\left(T,\|x-y\|_1\right).
        $$

    \smallskip\noindent\textup{ii)}
        If \eqref{nl-e07} holds with $u(s)=cs$ for some constant $c>0$, then for $T>0$, $x\in\real^d$ and all bounded Borel functions $f:\real^d\to\real$
        $$
            |\nabla P_Tf|^2(x)
            \leq
            \left\{P_Tf^2(x)-\big(P_Tf(x)\big)^2\right\}
            \frac{Cd}{T^{\kappa_3}\wedge T^{\kappa_4}}
            \left(1+c\int_0^T k(s)\,\eup^{cK(s)}\,\dup s\right)^2.
        $$
\end{corollary}

\begin{ack}
    We thank two anonymous referees for their
    critical comments which helped
    us to improve the presentation
    of our paper.
\end{ack}


\begin{thebibliography}{99}

\bibitem{ATW14}
M.\ Arnaudon, A.\ Thalmaier, F.-Y.\ Wang:
Equivalent Harnack and gradient inequalities for pointwise curvature lower bound.
\emph{Bull.\ Sci.\ Math.} \textbf{138} (2014) 643--655.


\bibitem{BHOZ08}
F.\ Biagini, Y.\ Hu, B.\ {\O}ksendal, T.\ Zhang:
\emph{Stochastic Calculus for Fractional Brownian Motion and Applications}.
Springer, London 2008.



\bibitem{Bi}
I.\ Bihari: A generalization of a lemma of Bellman and its application to uniqueness problems of differential equations.
\emph{Acta Math. Hung.} \textbf{7} (1956) 81--94.





\bibitem{DRW09}
G.\ Da\ Prato, M.\ R\"{o}ckner, F.-Y.\ Wang:
Singular stochastic equations on Hilbert spaces: Harnack inequalities for their transition semigroups.
\emph{J.\ Funct.\ Anal.} \textbf{257} (2009) 992--1017.


\bibitem{DU98}
L.\ Decreusefond, A.\ S.\ \"{U}st\"{u}nel:
Stochastic analysis of the fractional Brownian motion.
\emph{Potential\ Anal.} \textbf{10} (1999) 177--214.


\bibitem{Den14} C.-S.\ Deng:
Harnack inequalities for SDEs driven by subordinate Brownian motions.
\emph{J.\ Math.\ Anal.\ Appl.} \textbf{417} (2014) 970--978.

\bibitem{DS14}
C.-S.\ Deng, R.L.\ Schilling:
On shift Harnack inequalities for subordinate semigroups and moment estimates for L\'{e}vy processes.
\emph{Stoch.\ Proc.\ Appl.} \textbf{125} (2015) 3851--3878.





\bibitem{Fan13a}
X.-L.\ Fan:
Harnack inequality and derivative formula for SDE driven by fractional Brownian motion.
\emph{Sci.\ China\ Math.} \textbf{56} (2013) 515--524.


\bibitem{Fan14}
X.-L.\ Fan:
Harnack-type inequalities and applications for SDE driven by fractional Brownian motion.
\emph{Stoch.\ Anal.\ Appl.} \textbf{32} (2014) 602--618.


\bibitem{Fan13b}
X.-L.\ Fan, Y.\ Ren: Bismut formulas and applications for stochastic (functional) differential equations driven by fractional Brownian motions.
\emph{Stoch.\ Dyn.}
\textbf{17} (2017) 1750028, 19 pages,


\bibitem{GM14}
J.\ Gajda, M.\ Magdziarz:
Large deviations for subordinated Brownian motion and applications.
\emph{Statist.\ Probab.\ Lett.} \textbf{88} (2014) 149--156.

\bibitem{LS04}
W.\ Linde, Z.\ Shi:
Evaluating the small deviation probabilities for subordinated L\'{e}vy processes.
\emph{Stoch.\ Proc.\ Appl.} \textbf{113} (2004) 273--287.


\bibitem{MNX08}
M.M.\ Meerschaert, E.\ Nane, Y.\ Xiao:
Large deviations for local time fractional Brownian motion and applications.
\emph{J.\ Math.\ Anal.\ Appl.} \textbf{346} (2008) 432--445.



\bibitem{nourdin}
I.\ Nourdin:
\emph{Selected Aspects of Fractional Brownian Motion}.
Springer, Milan 2012.

\bibitem{nua-ouk02}
D.\ Nualart, Y.\ Ouknine:
Regularization of differential equations by fractional noise.
\emph{Stoch.\ Proc.\ Appl.} \textbf{102} (2002) 103--116.



\bibitem{RW10}
M.\ R\"{o}ckner, F.-Y.\ Wang:
Log-Harnack inequality for stochastic differential equations in Hilbert spaces and its consequences.
\emph{Inf.\ Dim.\ Anal.\ Quantum Probab.\ Rel.\ Top.} \textbf{13} (2010) 27--37.


\bibitem{SKM93}
S.G.\ Samko, A.A.\ Kilbas, O.I.\ Marichev:
\emph{Fractional Integrals and Derivatives, Theory and
Applications}. Gordon and Breach Science Publishers, 1993.


\bibitem{SSV}
R.L.\ Schilling, R.\ Song, Z.\ Vondra\v{c}ek:
\emph{Bernstein Functions. Theory and Applications (2nd Edn)}. De Gruyter, Studies in Mathematics \textbf{37}, Berlin 2012.



\bibitem{Wan97}
F.-Y.\ Wang:
Logarithmic Sobolev inequalities on noncompact Riemannian manifolds.
\emph{Probab.\ Theory\ Related\ Fields} \textbf{109} (1997) 417--424.


\bibitem{Wbook}
F.-Y.\ Wang:
\emph{Harnack Inequalities for Stochastic Partial Differential Equations}.
Springer, New York 2013.


\bibitem{WW14}
F.-Y.\ Wang, J.\ Wang:
Harnack inequalities for stochastic equations driven by L\'evy noise.
\emph{J.\ Math.\ Anal.\ Appl.} \textbf{410} (2014) 513--523.


\bibitem{WZ15}
L.\ Wang, X.\ Zhang:
Harnack inequalities for SDEs driven by cylindrical $\alpha$-stable processes.
\emph{Potential\ Anal.} \textbf{42} (2015) 657--669.


\bibitem{Zha13}
X.\ Zhang:
Derivative formula and gradient estimates for SDEs driven by $\alpha$-stable processes.
\emph{Stoch.\ Proc.\ Appl.} \textbf{123} (2013) 1213--1228.

\end{thebibliography}
\end{document}